\documentclass[12pt]{article}

\usepackage{amsmath}
\usepackage{amssymb,amsfonts}
\usepackage{graphicx}
\usepackage{times,amssymb,amscd}
\usepackage{verbatim}
\usepackage{enumerate}
\usepackage{subcaption}

\DeclareMathOperator\arctanh{arctanh}

\usepackage{lscape}

\usepackage{epstopdf}

\newcommand{\bh}{\mathbf{h}}
\newcommand{\ba}{\mathbf{a}}

\newcommand{\Ba}{\boldsymbol{a}}
\newcommand{\Bb}{\boldsymbol{b}}

\newcommand{\Bu}{\boldsymbol{u}}

\newcommand{\cD}{\mathcal{D}}
\newcommand{\cF}{\mathcal{F}}

\newcommand{\cT}{\mathcal{T}}
\newcommand{\cC}{\mathcal{C}}
\newcommand{\cB}{\mathcal{B}}

\newtheorem{theorem}{Theorem}
\newtheorem{corollary}{Corollary}
\newtheorem{lemma}{Lemma}
\newtheorem{defn}{Definition}
\newtheorem{rmrk}{Remark}
\newtheorem{proposition}{Proposition}

\newenvironment{remark}{\begin{rmrk}\normalfont}{\end{rmrk}}

\newenvironment{proof}{\paragraph{Proof:}}{\hfill$\square$}

\usepackage[usenames,dvipsnames]{xcolor} 
\usepackage{tikz, tikz-3dplot, pgfplots}
\usepackage{tkz-graph}
\usetikzlibrary[positioning,patterns] 

\usepackage{pictexwd,dcpic}

\usepackage{lscape}

\begin{document}
\pagestyle{myheadings}
\markboth{\centerline{R.~T.~Kozma and J.~Szirmai}}
{New Horoball Packing Density Lower Bound in Hyperbolic 5-space}
\title
{New Horoball Packing Density Lower Bound in Hyperbolic 5-space}

\author{\medbreak \medbreak {\normalsize{}} \\
\normalsize Robert T. Kozma$^{(1),(2)}$ ~and Jen\H{o}  Szirmai$^{(2)}$ \\
\normalsize (1) Department of Mathematics, Statistics, and Computer Science\\
\normalsize University of Illinois at Chicago \\
\normalsize Chicago IL 60607 USA \\
\normalsize (2) Budapest University of Technology and Economics\\
\normalsize Institute of Mathematics, Department of Geometry \\
\normalsize H-1521 Budapest, Hungary \\
\normalsize E-mail: rkozma2@uic.edu,~szirmai@math.bme.hu }

\date{}

\maketitle

\begin{abstract}
We describe the optimal horoball packings of asymptotic Koszul type Coxeter simplex tilings of $5$-dimensional hyperbolic space where the symmetries of the packings are generated by Coxeter groups. We find that the optimal horoball packing density of $\delta_{opt}=0.59421\dots$ is realized in an entire  commensurability class of arithmetic Coxeter tilings. Eleven optimal arrangements are achieved by placing horoballs at the asymptotic vertices of the corresponding Coxeter simplices that give the fundamental domains. When multiple horoball types are allowed, in the case of the arithmetic Coxeter groups, the relative packing densities of the optimal horoball types are rational submultiples of $\delta_{opt}$, corresponding to the Dirichlet-Voronoi cell densities of the packing. The packings given in this paper are so far the densest known in hyperbolic $5$-space.

\end{abstract}


\section{Introduction}

Let $X$ denote a space of constant curvature, either the $n$-dimensional sphere $\mathbb{S}^n$, 
Euclidean space $\mathbb{E}^n$, or 
hyperbolic space $\mathbb{H}^n$ with $n \geq 2$. An important question of discrete geometry is to find the highest possible packing density in $X$ by congruent non-overlapping balls of a given radius \cite{Be}, \cite{G--K}. 
The definition of packing density is crucial in hyperbolic space as shown by B\"or\"oczky \cite{B78}, for standard examples also see \cite{G--K}, \cite{R06}. 
The most widely accepted notion of packing density considers the local densities of balls with respect to their Dirichlet--Voronoi cells (cf. \cite{B78} and \cite{K98}). In order to consider horoball packings in $\overline{\mathbb{H}}^n$, we use an extended notion of such local density. 

Let $B$ be a horoball in packing $\cB$, and $P \in \overline{\mathbb{H}}^n$ be an arbitrary point. 
Define $d(P,B)$ to be the perpendicular distance from point $P$ to the horosphere $S = \partial B$, where $d(P,B)$ 
is taken to be negative when $P \in B$. The Dirichlet--Voronoi cell $\cD(B,\cB)$ of a horoball $B$ is defined as the convex body
\begin{equation}
\cD(B,\cB) = \{ P \in \mathbb{H}^n | d(P,B) \le d(P,B'), ~ \forall B' \in \cB \}. \notag
\end{equation}
Both $B$ and $\cD$ are of infinite volume, so the usual notion of local density is
modified as follows. Let $Q \in \partial{\mathbb{H}}^n$ denote the ideal center of $B$ at infinity, and take its boundary $S$ to be the one-point compactification of Euclidean $(n - 1)$-space.
Let $B_C^{n-1}(r) \subset S$ be the Euclidean $(n-1)$-ball with center $C \in S \setminus \{Q\}$.
Then $Q \in \partial {\mathbb{H}^n}$ and $B_C^{n-1}(r)$ determine a convex cone 
$\cC^n(r) = cone_Q\left(B_C^{n-1}(r)\right) \in \overline{\mathbb{H}}^n$ with
apex $Q$ consisting of all hyperbolic geodesics passing through $B_C^{n-1}(r)$ with limit point $Q$. The local density $\delta_n(B, \cB)$ of $B$ to $\cD$ is defined as
\begin{equation}
\delta_n(\cB, B) =\varlimsup\limits_{r \rightarrow \infty} \frac{vol(B \cap \cC^n(r))} {vol(\cD \cap \cC^n(r))}. \notag
\end{equation}
This limit is independent of the choice of center $C$ for $B^{n-1}_C(r)$.

In the case of periodic ball or horoball packings, the local density defined above can be extended to the entire hyperbolic space. This local density
is related to the simplicial density function (defined below) that we generalized in \cite{Sz12} and \cite{Sz12-2}.
In this paper we will use such definition of packing density (cf. Section 3).  

The alternate method suggested by Bowen and Radin \cite{Bo--R}, \cite{R06} uses Nevo's 
point-wise ergodic theorem 
to assure that the standard Euclidean limit notion of density is 
well-defined for $\mathbb{H}^n$. First they define a metric on 
the space $\Sigma_{\mathcal{P}}$ of relatively-dense packings by compact objects, based on Hausdorff distance, corresponding to uniform convergence on compact subsets of $\mathbb{H}^n$. Then they study the measures invariant under isometries of $\Sigma_{\mathcal{P}}$ rather than individual packings. There is a large class of packings of compact objects in hyperbolic space for which such density is well-defined. 
Using ergodic methods, 
they show that if there is only one optimally dense packing of $\mathbb{E}^n$ or $\mathbb{H}^n$, up to congruence, by congruent copies of bodies from some fixed finite collection, then that packing must have a symmetry group with compact fundamental domain. Moreover, for almost any radius $r \in [0,\infty)$ the optimal ball packing in $\mathbb{H}^n$ has low symmetry. 

A Coxeter simplex is an $n$-dimensional simplex in $X$ with dihedral angles either submultiples of $\pi$ or zero. 
The group generated by reflections on the sides of a Coxeter simplex is called a Coxeter simplex reflection group. 
Such reflections give a discrete group of isometries of $X$ with the Coxeter simplex as its fundamental domain; 
hence the groups give regular tessellations of $X$ if the fundamental simplex is characteristic. The Coxeter groups are finite for $\mathbb{S}^n$, and infinite for $\mathbb{E}^n$ or $\mathbb{H}^n$. 

In $\mathbb{H}^n$ we have noncompact simplices with ideal vertices at infinity $\partial \mathbb{H}^n$. 
Coxeter simplices exist only for dimensions $n = 2, 3, \dots, 9$; furthermore, only a finite number exist in dimensions $n \geq 3$. 
Johnson {\it et al.} \cite{JKRT} computed the volumes of all Coxeter simplices in hyperbolic $n$-space, also see Kellerhals \cite{K91}. 
Such simplices are the most elementary building blocks of hyperbolic manifolds,
the volume of which is an important topological invariant. 

In the $n$-dimensional space $X$ of constant curvature
 $(n \geq 2)$, define the simplicial density function $d_n(r)$ to be the density of $n+1$ spheres
of radius $r$ mutually tangent to one another with respect to the simplex spanned by the centers of the spheres. L.~Fejes T\'oth and H.~S.~M.~Coxeter
conjectured that the packing density of balls of radius $r$ in $X$ cannot exceed $d_n(r)$.
Rogers \cite{Ro64} proved this conjecture in Euclidean space $\mathbb{E}^n$.
The $2$-dimensional spherical case was settled by L.~Fejes T\'oth \cite{FTL}, and B\"or\"oczky \cite{B78}, who proved the following extension:
\begin{theorem}[K.~B\"or\"oczky]
In an $n$-dimensional space of constant curvature, consider a packing of spheres of radius $r$.
In the case of spherical space, assume that $r<\frac{\pi}{4}$.
Then the density of each sphere in its Dirichlet--Voronoi cell cannot exceed the density of $n+1$ spheres of radius $r$ mutually
touching one another with respect to the simplex spanned by their centers.
\end{theorem}
In hyperbolic space, 
the monotonicity of $d_3(r)$ was proved by B\"or\"oczky and Florian
in \cite{B--F64}; in \cite{Ma99} Marshall 
showed that for sufficiently large $n$, 
function $d_n(r)$ is strictly increasing in variable $r$. Kellerhals \cite{K98} showed $d_n(r)<d_{n-1}(r)$, and that in cases considered by Marshall the local density of each ball in its Dirichlet--Voronoi cell is bounded above by the simplicial horoball density $d_n(\infty)$.

This upper bound for density in hyperbolic space $\mathbb{H}^3$ is $0.85327\dots$, 
which is not realized by packing regular balls. However, it is attained by a horoball packing of
$\overline{\mathbb{H}}^3$ where the ideal centers of horoballs lie on an
absolute figure of $\overline{\mathbb{H}}^3$;
for example, they may lie at the vertices of the ideal regular
simplex tiling with Coxeter-Schl\"afli symbol $[3,3,6]$. 

In \cite{KSz} we proved that the optimal ball packing arrangement in $\mathbb{H}^3$ mentioned above is not unique. We gave several new examples of horoball packing arrangements based on totally asymptotic Coxeter tilings that yield the B\"or\"oczky--Florian upper bound \cite{B--F64}.
 
Furthermore, in \cite{Sz12}, \cite{Sz12-2} we found that 
by allowing horoballs of different types at each vertex of a totally asymptotic simplex and generalizing 
the simplicial density function to $\mathbb{H}^n$ for $(n \ge 2)$,
 the B\"or\"oczky-type density 
upper bound is no longer valid for the fully asymptotic simplices for $n \ge 3$. 
For example, in $\overline{\mathbb{H}}^4$ the locally optimal packing density is $0.77038\dots$, higher than the B\"or\"oczky-type density upper bound of $0.73046\dots$. 
However these ball packing configurations are only locally optimal and cannot be extended to the entirety of the
hyperbolic spaces $\mathbb{H}^n$. Further open problems and conjectures on $4$-dimensional hyperbolic packings are discussed in \cite{G--K--K}. 
Using horoball packings in $\mathbb{H}^4$, allowing horoballs of different types,
we found seven counterexamples (realized by allowing up to three horoball types) 
to one of L. Fejes T\'oth's conjectures stated in his foundational book Regular Figures.

The second-named author has several additional results on globally and locally optimal ball packings 
in $\mathbb{H}^n$, $\mathbb{S}^n$, and 
the eight Thurston geomerties arising from Thurston's geometrization conjecture 
\cite{Sz05-2}, \cite{Sz07-1}, \cite{Sz07-2}, \cite{Sz10}, \cite{Sz13-1}, \cite{Sz13-2}. 
 
In this paper we continue our investigations 
of ball packings, in hyperbolic 5-space. 
Using horoball packings, allowing horoballs of different types when applicable,
we find the packing densities of the horoballs with respect to the Coxeter simplex cells.

\section{The Cayley--Klein model of  $n$-dimensional hyperbolic geometry}

In this paper we use the Cayley--Klein model, and a projective interpretation of hyperbolic geometry. This has the advantage of greatly 
simplifying our calculations in higher dimensions as compared to other models such as the conformal Poincar\'e model. In the Klein model, hyperbolic objects are straight or convex if and only if they are straight or convex in the embedded model. Hyperbolic symmetries are then modeled as Euclidean projective transformations using the projective linear group $PGL(n+1,\mathbb{R})$.
In this section we give a brief review of the concepts used in this paper. For a general discussion and background in hyperbolic geometry 
and the projective models of Thurston geometries see \cite{Mol97} and \cite{MSz}.

\subsection{The Projective Model}

We use the projective model in Lorentzian $(n+1)$-space
$\mathbb{E}^{1,n}$ of signature $(1,n)$, i.e. $\mathbb{E}^{1,n}$ is
the real vector space $\mathbf{V}^{n+1}$ equipped with the bilinear
form of signature $(1,n)$,
$\langle \mathbf{x}, \mathbf{y} \rangle = -x^0y^0+x^1y^1+ \dots + x^n y^n \label{bilinear_form}$
where the non-zero real vectors 
$\mathbf{x}=(x^0,x^1,\dots,x^n)\in\mathbf{V}^{n+1}$ 
and $ \mathbf{y}=(y^0,y^1,\dots,y^n)$ $\in\mathbf{V}^{n+1}$ represent points in projective space 
$\mathcal{P}^n(\mathbb{R})$. $\mathbb{H}^n$ is represented as the
interior of the absolute quadratic form
$Q=\{\mathbf{x}\in\mathcal{P}^n | \langle \mathbf{x}, \mathbf{x} \rangle =0 \}=\partial \mathbb{H}^n$
in real projective space $\mathcal{P}^n(\mathbf{V}^{n+1},
\boldsymbol{V}_{n+1})$. All proper interior points $\mathbf{x} \in \mathbb{H}^n$ satisfy
$\langle \mathbf{x}, \mathbf{x} \rangle < 0$.

The boundary points $\partial \mathbb{H}^n $ in
$\mathcal{P}^n$ represent the absolute points at infinity of $\mathbb{H}^n$.
Points $\mathbf{y}$ satisfying $\langle \mathbf{y}, \mathbf{y} \rangle >
0$ lie outside $\partial \mathbb{H}^n $ and are referred to as outer points
of $\mathbb{H}^n$. Take $P(\mathbf{x}) \in \mathcal{P}^n$, point
$\mathbf{y} \in \mathcal{P}^n$ is said to be conjugate to
$\mathbf{x}$ relative to $Q$ when $\langle
\mathbf{x}, \mathbf{y} \rangle =0$. The set of all points conjugate
to $P(\mathbf{x})$ form a projective (polar) hyperplane
$pol(P)=\{\mathbf{y}\in\mathcal{P}^n | \langle  \mathbf{x}, \mathbf{y} \rangle =0 \}.$
Hence the bilinear form $Q$ induces a bijection
or linear polarity $\mathbf{V}^{n+1} \rightarrow
\boldsymbol{V}_{n+1}$
between the points of $\mathcal{P}^n$
and its hyperplanes.
Point $X (\mathbf{x})$ and hyperplane $\alpha
(\Ba)$ are incident if the value of
the linear form $\Ba$ evaluated on vector $\mathbf{x}$ is
 zero, i.e. $\mathbf{x}\Ba=0$ where $\mathbf{x} \in \
\mathbf{V}^{n+1} \setminus \{\mathbf{0}\}$, and $\Ba \in
\mathbf{V}_{n
+1} \setminus \{\mathbf{0}\}$.
Similarly, lines in $\mathcal{P}^n$ are characterized by
2-subspaces of $\mathbf{V}^{n+1}$ or $(n-1)$-spaces of $\mathbf{V}_{n+1}$ \cite{Mol97}.

Let $P \subset \mathbb{H}^n$ denote a polyhedron bounded by
a finite set of hyperplanes $H^i$ with unit normal vectors
$\Bb^i \in \mathbf{V}_{n+1}$ directed
 towards the interior of $P$:
\begin{equation}
H^i=\{\mathbf{x} \in \mathbb{H}^d | \mathbf{x} \Bb^i =0 \} \ \ \text{with} \ \
\langle \Bb^i,\Bb^i \rangle = 1.
\end{equation}
In this paper $P$ is assumed to be an acute-angled polyhedron
 with proper or ideal vertices.
The Grammian matrix $G(P)=( \langle \Bb^i,
\Bb^j \rangle )_{i,j} ~ {i,j \in \{ 0,1,2 \dots n \} }$  is a
 symmetric matrix of signature $(1,n)$ with entries
$\langle \Bb^i,\Bb^i \rangle = 1$
and $\langle \Bb^i,\Bb^j \rangle
\leq 0$ for $i \ne j$ where

$$
\langle \mbox{\boldmath$b$}^i,\mbox{\boldmath$b$}^j \rangle =
\left\{
\begin{aligned}
&0 & &\text{if}~H^i \perp H^j,\\
&-\cos{\alpha^{ij}} & &\text{if}~H^i,H^j ~ \text{intersect \ along an edge of $P$ \ at \ angle} \ \alpha^{ij}, \\
&-1 & &\text{if}~\ H^i,H^j ~ \text{are parallel in the hyperbolic sense}, \\
&-\cosh{l^{ij}} & &\text{if}~H^i,H^j ~ \text{admit a common perpendicular of length} \ l^{ij}.
\end{aligned}
\right.
$$
This is summarized using the weighted graph or scheme of the polytope $\sum(P)$. The graph nodes correspond 
to the hyperplanes $H^i$ and are connected if $H^i$ and $H^j$ are not perpendicular ($i \neq j$).
If they are connected we write the positive weight $k$ where  $\alpha_{ij} = \pi / k$ on the edge,  
unlabeled edges denote an angle of $\pi/3$. For examples, see the Coxeter diagrams in Table \ref{table:simplex_list}.

In this paper we set the sectional curvature of $\mathbb{H}^n$,
$K=-k^2$, to be $k=1$. The distance $d$ of two proper points
$\mathbf{x}$ and $\mathbf{y}$ is calculated by the formula
\begin{equation}
\cosh{{d}}=\frac{-\langle ~ \mathbf{x},~\mathbf{y} \rangle }{\sqrt{\langle ~ \mathbf{x},~\mathbf{x} \rangle
\langle ~ \mathbf{y},~\mathbf{y} \rangle }}.
\label{prop_dist}
\end{equation}
The perpendicular foot $Y(\mathbf{y})$ of point $X(\mathbf{x})$ dropped onto plane $\Bu$ is given by
\begin{equation}
\mathbf{y} = \mathbf{x} - \frac{ \langle \mathbf{x}, \mathbf{u} \rangle }{\langle \Bu, \Bu \rangle} \mathbf{u},
\end{equation}
where $\Bu$ is the pole of the plane $\mathit{u}$.

\subsection{Horospheres and Horoballs in $\mathbb{H}^n$}

A horosphere in $\mathbb{H}^n$ ($n \ge 2)$ is a 
hyperbolic $n$-sphere with infinite radius centered 
at an ideal point on $\partial \mathbb{H}^n$. Equivalently, a horosphere is an $(n-1)$-surface orthogonal to
the set of parallel straight lines passing through a point of the absolute quadratic surface. 
A horoball is a horosphere together with its interior. 

In order to derive the equation of a horosphere, we introduce a projective 
coordinate system for $\mathcal{P}^n$ with a vector basis 
$\bold{a}_i \ (i=0,1,2,\dots, n)$ so that the Cayley-Klein ball model of $\mathbb{H}^n$ 
is centered at $(1,0,0,\dots, 0)$, and set an
arbitrary point at infinity to lie at $A_0=(1,0,\dots, 0,1)$. 
The equation of a horosphere with center
$A_0$ passing through point $S=(1,0,\dots,0,s)$ is derived from the equation of the 
the absolute sphere $-x^0 x^0 +x^1 x^1+x^2 x^2+\dots + x^n x^n = 0$, and the plane $x^0-x^n=0$ tangent to the absolute sphere at $A_0$. 
The general equation of the horosphere is
\begin{equation}
0=\lambda (-x^0 x^0 +x^1 x^1+x^2 x^2+\dots + x^n x^n)+\mu{(x^0-x^n)}^2.
\end{equation}
Plugging in for $S$ we obtain
\begin{equation}
\lambda (-1+s^2)+\mu {(-1+s)}^2=0 \text{~~and~~} \frac{\lambda}{\mu}=\frac{1-s}{1+s}. \notag
\end{equation}
For $s \neq \pm1$, the equation of a horosphere in projective coordinates is
\begin{align}
\label{eqn:horosphere}
(s-1)\left(-x^0 x^0 +\sum_{i=1}^n (x^i)^2\right)-(1+s){(x^0-x^n)}^2 & =0.
\end{align}

In an $n$-dimensional hyperbolic space any two horoballs are congruent in the classical sense: each have an infinite radius.
However, it is often useful to distinguish between certain horoballs of a packing. 
We use the notion of horoball type with respect to the fundamental domain of the given tiling as introduced in \cite{Sz12-2}.

Two horoballs of a horoball packing are said to be of the {\it same type} or {\it equipacked} if 
and only if their local packing densities with respect to a given cell (in our case a 
Coxeter simplex) are equal. If this is not the case, then we say the two horoballs are of {\it different type}. 
For example, in the above discussion horoballs centered at $A_0$ passing through $S$ with different values for the final coordinate $s$ are of different type relative to 
an appropriate cell, all such horoballs yield a one-parameter family.

In order to compute volumes of horoball pieces, we use J\'anos Bolyai's classical formulas from the mid 19-th century:
\begin{enumerate}
\item 
The hyperbolic length $L(x)$ of a horospheric arc that belongs to a chord segment of length $x$ is
\begin{equation}
\label{eq:horo_dist}
L(x)=2 \sinh{\left(\frac{x}{2}\right)} .
\end{equation}
\item The intrinsic geometry of a horosphere is Euclidean, 
so the $(n-1)$-dimensional volume $\mathcal{A}$ of a polyhedron $A$ on the 
surface of the horosphere can be calculated as in $\mathbb{E}^{n-1}$.
The volume of the horoball piece $\mathcal{H}(A)$ determined by $A$ and 
the aggregate of axes 
drawn from $A$ to the center of the horoball is
\begin{equation}
\label{eq:bolyai}
vol(\mathcal{H}(A)) = \frac{1}{n-1}\mathcal{A}.
\end{equation}
\end{enumerate}

\section{Horoball packings of Coxeter Simplices with Ideal Verticies}

Let $\cT$ be a $5$-dimensional Coxeter tiling \cite{IH90}, \cite{JKRT2}. A rigid motion mapping one cell of $\cT$ onto
another maps the entire tiling onto itself. The symmetry group of a Coxeter tiling
 contains its Coxeter group, denoted by $\Gamma_\cT$. 
Any simplex cell of $\cT$ acts as a fundamental domain $\cF_{\cT}$
of $\Gamma_\cT$, where the Coxeter group is generated by reflections on the $(n - 1)$-dimensional facets of $\cF_{\cT}$. 
In this paper we consider only asymptotic Koszul Coxeter simplices, i.e. ones that have at least one ideal vertex, in this case the orbifold $\mathbb{H}^5/\Gamma_{\cT}$ has at least one cusp. In Table \ref{table:simplex_list} we list the twelve Coxeter simplices that exist in hyperbolic $5$-space, all of which are asymptotic Koszul type, together with their volumes. 
Note that volumes are given in terms Riemann's zeta function $\zeta(n)=\sum^{\infty}_{r=1}r^{-n}$, or the Dirichlet $L$-function $L(s,d)=\sum^{\infty}_{n=1}\left(\frac{n}{d}\right)n^{-s}$, where $(n/d)$ is the Legendre symbol. The volume of $\overline{P}_5$ was found by Monte-Carlo simulations \cite{JKRT}. 
All but $\widehat{AU}_5$ are arithmetic, the volume of $\widehat{AU}_5$ may be given as an integral equation, and no exact expression is known. For a complete discussion of hyperbolic Coxeter simplices 
and their volumes for dimensions $n \geq 3$, see Johnson {\it et al.} \cite{JKRT}. 

We define the density of a horoball packing $\mathcal{B}_{\cT}$ of a Coxeter simplex tiling $\cT$ as
\begin{equation}
\delta(\mathcal{B}_{\cT})=\frac{\sum_{i=1}^m vol(\mathcal{B}_i \cap \cF_{\cT})}{vol(\cF_{\cT})}.
\end{equation}
Here $\cF_{\cT}$ denotes the simplicial fundamental domain of tiling $\cT$, $m$ is the number of ideal vertices of $\cF_\cT$, 
and $\mathcal{B}_i$ are the horoballs centered at ideal vertices. 
We allow horoballs of different types at the asymptotic vertices of the tiling. 
A horoball type is allowed if it yields a packing: no two horoballs may have an interior point in common. 
In addition we require that no horoball extend beyond the facet opposite the vertex where it is centered so that the packing remains invariant under the actions of the Coxeter group of the tiling.
If these conditions are satisfied, we can extend the packing density from 
the simplicial fundamental domain $\cF_{\cT}$ to the entire $\mathbb{H}^5$ using the Coxeter group $\Gamma_{\tau}$.
In the case of Coxeter simplex tilings, Dirichlet--Voronoi cells coincide with the Coxeter simplices. We denote the optimal horoball packing density as
\begin{equation}
\delta_{opt}(\cT) = \sup\limits_{\mathcal{B}_{\cT} \text{~packing}} \delta(\mathcal{B}_{\tau}).
\end{equation}

\begin{table}
\resizebox{\columnwidth}{!}{%
    \begin{tabular}{|cc|c|c|c|}
    \hline
    Coxeter & ~ & Witt & Simplex & Optimal \\
    Diagram  & Notation & Symbol & Volume & Packing Density\\
    \hline
        Simply Asymptotic & ~ & ~ & ~ & ~\\
    \hline
    \begin{tikzpicture}
	
	\draw (0,0) -- (2.5,0);

	\draw[fill=black] (0,0) circle (.05);
	\draw[fill=black] (.5,0) circle (.05); 
	\draw[fill=black] (1,0) circle (.05);
	\draw[fill=black] (1.5,0) circle (.05);
	\draw[fill=black] (2,0) circle (.05);
	\draw  (2.5,0) circle (.05);

	\node at (1.75,0.175) {$4$};

\end{tikzpicture}
  & $[3,3,3,4,3]$ & $\overline{U}_5$ & $7 \zeta(3)/46080$ & $0.59421\dots$  \\

    \begin{tikzpicture}
	
	\draw (0,0) -- (1.975,0);
	\draw (1,0) -- (1,.5);

	\draw[fill=black]  (0,0) circle (.05);
	\draw[fill=black] (.5,0) circle (.05); 
	\draw[fill=black] (1,0) circle (.05);
	\draw[fill=black] (1.5,0) circle (.05);
	\draw (2,0) circle (.05);
	\draw [fill=black] (1,.5) circle (.05);

	\node at (.25,0.175) {$4$};

    \end{tikzpicture}
  & $[4,3,3^{2,1}]$ & $\overline{S}_5$ & $7 \zeta(3)/15360$ & $0.59421\dots$  \\

    \begin{tikzpicture}
	
	\draw (0.025,0) -- (1.5,0);
	\draw (1,-.5) -- (1,.5);

	\draw (0,0) circle (.05);
	\draw[fill=black] (.5,0) circle (.05); 
	\draw[fill=black] (1,0) circle (.05);
	\draw[fill=black] (1.5,0) circle (.05);
	\draw[fill=black] (1,.5) circle (.05);
	\draw[fill=black] (1,-.5) circle (.05);

    \end{tikzpicture}
  & $[3^{2,1,1,1}]$ & $\overline{Q}_5$ & $7 \zeta(3)/7680$ & $0.59421\dots$  \\

    \begin{tikzpicture}
	
	\draw (0,0) -- (.5,0);
	\draw (.5,0) -- (1,0.25);
	\draw (.5,0) -- (1,-0.25);
	\draw (1,.25) -- (1.5,0.25);
	\draw (1,-.25) -- (1.5,-0.25);
	\draw (1.5,.25) -- (1.5,-0.25);

	\draw (0,0) circle (.05);
	\draw[fill=black] (.5,0) circle (.05); 
	\draw[fill=black] (1,.25) circle (.05);
	\draw[fill=black] (1,-.25) circle (.05);
	\draw[fill=black] (1.5,0.25) circle (.05);
	\draw[fill=black] (1.5,-0.25) circle (.05);
	
\end{tikzpicture}

  & $[3, 3 ^{[5]}]$ & $\overline{P}_5$ & $5^{3/2}L(3,5)/4608$ & $0.56151\dots$  \\

\hline
    Doubly Asymptotic & ~ & ~ & ~ & ~\\
\hline
    \begin{tikzpicture}
	
	\draw (0.025,0) -- (2.475,0);

	\draw (0,0) circle (.05);
	\draw[fill=black] (.5,0) circle (.05); 
	\draw[fill=black] (1,0) circle (.05);
	\draw[fill=black] (1.5,0) circle (.05);
	\draw[fill=black] (2,0) circle (.05);
	\draw (2.5,0) circle (.05);

	\node at (1.25,0.175) {$4$};

\end{tikzpicture}
  & $[3,3,4,3,3]$ & $\overline{X}_5$ & $7 \zeta(3)/9216$ & $0.59421\dots$ \\
  
    \begin{tikzpicture}
	
	\draw (0.025,0) -- (2.475,0);

	\draw (0,0) circle (.05);
	\draw[fill=black] (.5,0) circle (.05); 
	\draw[fill=black] (1,0) circle (.05);
	\draw[fill=black] (1.5,0) circle (.05);
	\draw[fill=black] (2,0) circle (.05);
	\draw (2.5,0) circle (.05);

	\node at (.75,0.175) {$4$};
	\node at (2.25,0.175) {$4$};
	
\end{tikzpicture}
  & $[3,4,3,3,4]$ & $\overline{R}_5$ & $7 \zeta(3)/4608$ & $0.59421\dots$  \\
  
    \begin{tikzpicture}

	\draw (.5,.25) -- (.5,-.25);	
	\draw (.5,.25) -- (1,.25);
	\draw (.5,-.25) -- (1,-.25);
	\draw (1,.25) -- (1.5,.25);
	\draw (1,-.25) -- (1.5,-.25);
	\draw (1.5,.25) -- (1.5,-.25);
	
	\draw[fill=black] (.5,.25) circle (.05); 
	\draw[fill=black] (.5,-.25) circle (.05); 
	\draw[fill=black] (1,.25) circle (.05);
	\draw[fill=black] (1,-.25) circle (.05);
	\draw (1.5,0.25) circle (.05);
	\draw (1.5,-0.25) circle (.05);
	
	\node at (.375,0) {$4$};
	
\end{tikzpicture}
  & $[(3^5,4)]$ & $\widehat{AU}_5$ & $0.0075726186$ & $0.50108\dots$ \\

\hline
    Triply Asymptotic & ~ & ~ & ~ & ~\\
\hline
    \begin{tikzpicture}
	
	\draw (0,0) -- (2,0);
	\draw (1,0) -- (1,.5);

	\draw (0,0) circle (.05);
	\draw[fill=black] (.5,0) circle (.05); 
	\draw[fill=black] (1,0) circle (.05);
	\draw[fill=black] (1.5,0) circle (.05);
	\draw (2,0) circle (.05);
	\draw (1,.5) circle (.05);

	\node at (.25,0.175) {$4$};
	\node at (1.75,0.175) {$4$};

    \end{tikzpicture}
  & $[4,3,_{3}^{3,4}]$ & $\overline{N}_5$ & $7 \zeta(3)/1536$ & $0.59421\dots$  \\\begin{tikzpicture}
	
	\draw (-.5,0) -- (1,0);
	\draw (1,0) -- (1.5,0.25);
	\draw (1,0) -- (1.5,-0.25);

	\draw (-.5,0) circle (.05);	
	\draw[fill=black] (0,0) circle (.05);
	\draw[fill=black] (0.5,0) circle (.05);
	\draw[fill=black] (1,0) circle (.05);
	\draw (1.5,0.25) circle (.05);
	\draw (1.5,-0.25) circle (.05);
	
	\node at (.25,0.25) {$4$};
	
\end{tikzpicture}
  & $[3,4,3,3^{1,1}]$ & $\overline{O}_5$ & $7 \zeta(3)/2304$ & $0.59421\dots$  \\
  
\hline
    4-asymptotic & ~ & ~ & ~ & ~\\
\hline
    \begin{tikzpicture}
	
	\draw (0,0) -- (1.5,0);
	\draw (1,-.5) -- (1,.5);

	\draw (0,0) circle (.05);
	\draw[fill=black] (.5,0) circle (.05); 
	\draw[fill=black] (1,0) circle (.05);
	\draw (1.5,0) circle (.05);
	\draw (1,.5) circle (.05);
	\draw (1,-.5) circle (.05);
	
	\node at (.25,0.25) {$4$};

    \end{tikzpicture}
  & $[4,3,3^{1,1,1}]$ & $\overline{M}_5$ & $7 \zeta(3)/768$ & $0.59421\dots$  \\
\hline
    5-asymptotic & ~ & ~ & ~ & ~\\
\hline
    \begin{tikzpicture}

	\draw (0,0) -- (0,0.5);
	\draw (0,0) -- (-0.29,-0.4);	
	\draw (0,0) -- (0.5,0.15);	
	\draw (0,0) -- (-0.475,0.15);	
	\draw (0,0) -- (0.29,-0.4);	

	\draw[fill=black] (0,0) circle (.05);	
	\draw (0,0.5) circle (.05);
	\draw (-0.29,-0.4) circle (.05);
	\draw (0.5,0.15) circle (.05);
	\draw (-0.475,0.15) circle (.05);
	\draw (0.29,-0.4) circle (.05);

\end{tikzpicture}
  & $[3^{1,1,1,1,1}]$ & $\overline{L}_5$ & $7 \zeta(3)/384$ & $0.59421\dots$  \\
\hline
    Totally Asymptotic (6-asymptotic) & ~ & ~ & ~ & ~\\
\hline
    \begin{tikzpicture}

	\draw (.5,.25) -- (.5,-.25);	
	\draw (.5,.25) -- (1,.25);
	\draw (.5,-.25) -- (1,-.25);
	\draw (1,.25) -- (1.5,.25);
	\draw (1,-.25) -- (1.5,-.25);
	\draw (1.5,.25) -- (1.5,-.25);
	
	\draw (.5,.25) circle (.05); 
	\draw (.5,-.25) circle (.05); 
	\draw (1,.25) circle (.05);
	\draw (1,-.25) circle (.05);
	\draw (1.5,0.25) circle (.05);
	\draw (1.5,-0.25) circle (.05);
	
	\node at (.375,0) {$4$};
	\node at (1.625,0) {$4$};
	
\end{tikzpicture}

  & $[(3^2,4)^{[2]}]$ & $\widehat{UR}_5$ & $7 \zeta(3)/288$ & $0.59421\dots$  \\
\hline

    \end{tabular}%
    }
    \caption{Notation and volumes for the twelve asymptotic Coxeter Simplices in $\mathbb{H}^5$, in the Coxeter diagram empty circles denote reflection planes opposite an ideal vertex.}
    \label{table:simplex_list}
\end{table}

\begin{figure}
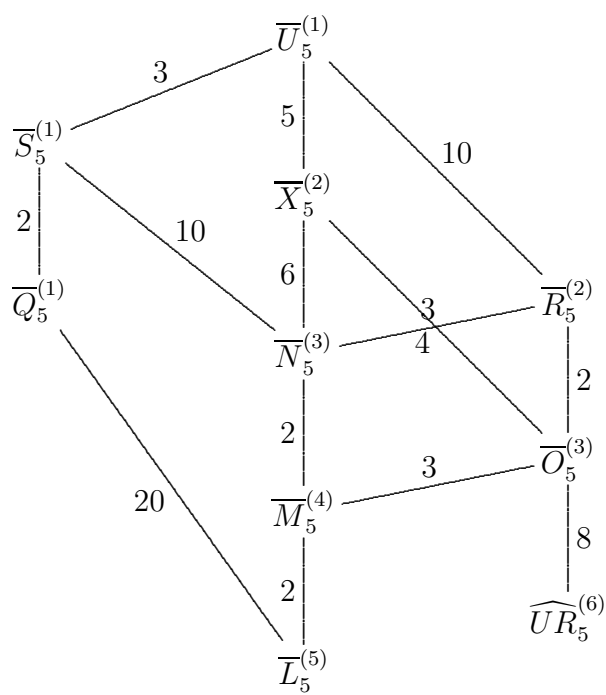

\begindc{\commdiag}[200]
\obj(5,10)[uu]{$\overline{U}_5^{(1)}$}
\obj(0,8)[ss]{$\overline{S}_5^{(1)}$}
\obj(0,5)[qq]{$\overline{Q}_5^{(1)}$}
\obj(5,7)[xx]{$\overline{X}_5^{(2)}$}
\obj(5,4)[nn]{$\overline{N}_5^{(3)}$}
\obj(5,1)[mm]{$\overline{M}_5^{(4)}$}
\obj(5,-2)[ll]{$\overline{L}_5^{(5)}$}
\obj(10,5)[rr]{$\overline{R}_5^{(2)}$}
\obj(10,2)[oo]{$\overline{O}_5^{(3)}$}
\obj(10,-1)[ur]{$\widehat{UR}_5^{(6)}$}
\mor{uu}{ss}{$3$}[\atright, \solidline]
\mor{ss}{qq}{$2$}[\atright, \solidline]
\mor{uu}{xx}{$5$}[\atright, \solidline]
\mor{xx}{nn}{$6$}[\atright, \solidline]
\mor{nn}{mm}{$2$}[\atright, \solidline]
\mor{mm}{ll}{$2$}[\atright, \solidline]
\mor{uu}{rr}{$10$}[\atleft, \solidline]
\mor{rr}{oo}{$2$}[\atleft, \solidline]
\mor{oo}{ur}{$8$}[\atleft, \solidline]
\mor{ss}{nn}{$10$}[\atleft, \solidline]
\mor{qq}{ll}{$20$}[\atright, \solidline]
\mor{rr}{nn}{$3$}[\atright, \solidline]
\mor{oo}{mm}{$3$}[\atright, \solidline]
\mor{xx}{oo}{$4$}[\atright, \solidline]
\enddc
\caption{Lattice of subgroups of the ten commensurable cocompact Coxeter groups in $\mathbb{H}^5$. The number in the superscript indicates number of ideal vertices of the fundamental simplex.}
\label{fig:lattice_of_subgroups}
\end{figure}

\subsection{Simply Asymptotic Cases}

We next give an example of the computation of the optimal horoball packing density through the Coxeter simplex tiling $\overline{U}_5$, for the other simply asymptotic cases $\overline{S}_5, \overline{Q}_5$ and $\overline{P}_5$  we apply the same method.

\begin{proposition}
\label{proposition:s4}
The optimal horoball packing density for simply asymptotic Coxeter simplex tiling $\cT_{{\overline{U}}_5}$ is $\delta_{opt}({\overline{U}}_5)
= 0.59421\dots$.
\end{proposition}

\begin{proof}
Let $\cF_{{\overline{U}_5}}$ be the simplicial fundamental domain of Coxeter tiling $\cT_{\overline{U}_5}$, with vertices $A_0, A_1, \dots, A_5$. Their coordinates are determined by the angle requirements. 
Our choice of vertices, as well as forms (`roots') for hyperplanes $\Bu_i$ opposite to vertices $A_i$,  
are given in Table \ref{table:data_one}.
In order to maximize the packing density, we determine the largest horoball type $\mathcal{B}_0(s)$ centered at ideal vertex $A_0$ that is admissible in cell $\cF_{{\overline{U}}_5}$. This extremal horoball is a member of the 1-parameter family of horoballs centered at $A_0$ with type-parameter 
$s$ (intuitively this can be thought of as the ``radius" of the horoball, the minimal Euclidean signed distance to our choice of center for the model, negative if the horoball contains the model center) such that the horoball $\mathcal{B}_0(s)$ is tangent to 
the plane of the hyperface $\Bu_0$ bounding 
the fundamental simplex opposite of $A_0$. By a projection, the perpendicular foot $F_0(\mathbf{f}_0)$ of vertex $A_0$ on plane $\Bu_0$,
\begin{equation}
\mathbf{f}_0 =\ba_0 - \frac{\langle \ba_0, \mathbf{u}_0 \rangle}{\langle \Bu_0,\Bu_0 \rangle} 
\mathbf{u}_0 = \left(1,0,0,0,0\right)=A_1,
\end{equation}
is the point of tangency of horoball $\mathcal{B}_0(s)$ and hyperface $\Bu_0$ of the the simplex cell.

Plugging in for $F_0$ and solving the horosphere equation, we find that the horoball with type-parameter
 $s=0$ is the optimal type. 
The equation of horosphere $\partial \mathcal{B}_0 =\partial \mathcal{B}_0(0)$ centered at $A_0$ passing through $F_0$ is

\begin{equation}
2 \left(h_1^2+h_2^2+h_3^2+h_4^2\right)+4 \left(h_5-\frac{1}{2}\right)^2=1.
\end{equation}

\begin{figure}[h]
\begin{center}
\includegraphics[height=50mm]{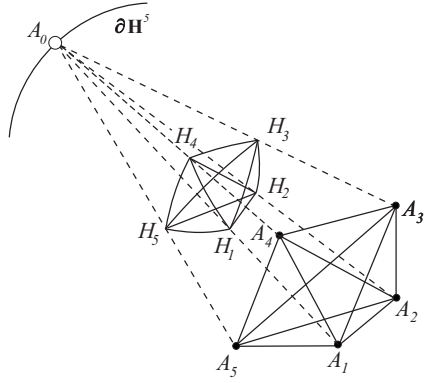}
\end{center}
\caption{Simply asymptotic horospheric 5-simplex on hyperface opposite $A_0$. Horoball $\mathcal{B}_0$ intersects the sides of the simplex at $H_i$.}
\label{fig:horospheric_triangle}
\end{figure} 

The intersections $H_i(\bh_i)$ of horosphere $\partial \mathcal{B}_0$ and simplex edges are found by parameterizing the 
simplex edges as $\bh_i(\lambda) = \lambda \ba_0 +\ba_i$ $(i=1,2,\dots,5)$, and computing their intersections with $\partial \mathcal{B}_0$. 
See Fig. \ref{fig:horospheric_triangle} and Table \ref{table:data_one} for the intersection points.
The volume of the horospherical 4-simplex determines the volume of the horoball piece by equation \eqref{eq:bolyai}.
In order to determine the data of the horospheric 4-simplex, we compute the hyperbolic distances $l_{ij}$ by the equation \ref{prop_dist},
$l_{ij} = d(H_i, H_j)$ where $d(\bh_i,\bh_j)= \arccos\left(\frac{-\langle \bh_i, \bh_j 
\rangle}{\sqrt{\langle \bh_i, \bh_i \rangle \langle \bh_j, \bh_j \rangle}}\right)$.
Moreover, the horospherical distances $L_{ij}$ can be calculated by formula \eqref{eq:horo_dist}.
The intrinsic geometry of the horosphere is Euclidean, so we use the 
Cayley-Menger determinant to find the volume $\mathcal{A}$ of the horospheric 4-simplex $\mathcal{A}$,

\begin{equation}
\mathcal{A} = \frac{1}{9216}
\begin{vmatrix}
 0 & 1 & 1 & 1 & 1 & 1\\
 1 & 0 & L_{12}^2 & L_{13}^2 & L_{14}^2 & L_{15}^2\\
 1 & L_{12}^2 & 0 & L_{23}^2 & L_{24}^2 & L_{25}^2\\
 1 & L_{13}^2 & L_{23}^2 & 0 & L_{34}^2 & L_{35}^2\\
 1 & L_{14}^2 & L_{24}^2 & L_{34}^2 & 0 & L_{45}^2\\
  1 & L_{15}^2 & L_{25}^2 & L_{35}^2 & L_{45}^2 & 0
 \end{vmatrix} = 0.00043\dots
 \end{equation}

The volume of the optimal horoball piece contained in the fundamental simplex is

\begin{equation}
vol(\mathcal{B}_0 \cap \cF_{\overline{U}_5}) = \frac{1}{n-1}\mathcal{A} = \frac{1}{4} \cdot 0.00043\dots = 0.00010\dots
\end{equation}

Hence by the Coxeter group $\Gamma_{\overline{U}_5}$ the optimal horoball packing density of the Coxeter Simplex tiling $\cT_{\overline{U}_5}$ becomes

\begin{equation}
\delta_{opt}(\overline{U}_5) = \frac{vol(\mathcal{B}_0 \cap \cF_{\overline{U}_5})}{vol(\cF_{\overline{U}_5})}\\
    = \frac{0.00010\dots}{7 \zeta(3)/46080} \\
    = 0.59421\dots 
\end{equation}
\end{proof}

The same method is used to find the optimal packing density of the remaining simply asymptotic Coxeter simplex tilings. Results of the computations are given  
in Table \ref{table:data_one}. We summarize the results:

\begin{theorem}
The optimal horoball packing density for the arithmetic simply asymptotic Coxeter simplex tiling $\cT_\Gamma$, $\Gamma \in \Big\{ \overline{U}_5, 
\overline{S}_5, \overline{Q}_5 \Big\}$ is $\delta_{opt}(\Gamma) = 0.59421\dots$, and in the tiling corresponding to the non-commensurable arithmetic Coxeter simplex group $\cT_{\overline{P}_5}$ it is $\delta_{opt}(\Gamma) = 0.56151\dots$.
\end{theorem}

\begin{table}[h!]
\resizebox{\columnwidth}{!}{%
	\begin{tabular}{|l|l|l|l|l|}
		 \hline
		 \multicolumn{5}{|c|}{{\bf Coxeter Simplex Tilings} }\\
		\hline
		 Witt Symb. & $\overline{U}_5$ &  $\overline{S}_5$ &  $\overline{Q}_5$ &  $\overline{P}_5$ \\
		 \hline
		 \multicolumn{5}{|c|}{{\bf Vertices of Simplex} }\\
		 \hline
		 $A_0$ & $(1, 0, 0, 0, 0, 1)*$ & $(1, 0, 0, 0, 0, 1)*$ & $(1, 0, 0, 0, 0, 1)*$ & $(1, 0, 0, 0, 0, 1)*$  \\
		 $A_1$ & $(1, 0, 0, 0, 0, 0)$ &$(1,0,0,0,0,0)$&$(1,0,0,0,0,0)$&$(1,0,0,0,0,0)$  \\
		 $A_2$ & $(1,\frac{1}{2},0,0,0,0)$ & $(1,0,0,0,\frac{1}{\sqrt{2}},0)$ & $(1,0,0,0,-\frac{1}{2},0)$ & $(1,0,0,0,\sqrt{\frac{2}{5}},0)$   \\
		 $A_3$ & $(1,\frac{1}{2},\frac{\sqrt{3}}{6},0,0,0)$ & $(1,0,0,\frac{\sqrt{6}}{8},\frac{3 \sqrt{2}}{8},0)$ & $(1,0,0,\frac{1}{2},-\frac{1}{2},0)$ & $(1,0,0,\frac{\sqrt{6}}{4},\frac{3}{2 \sqrt{10}},0)$  \\
		 $A_4$ & $(1,\frac{1}{2},\frac{\sqrt{3}}{6},\frac{\sqrt{6}}{12},0,0)$ & $(1,0,\frac{\sqrt{3}}{6},\frac{\sqrt{6}}{12},\frac{\sqrt{2}}{4},0)$ & $(1,0,-\frac{1}{2},0,-\frac{1}{2},0)$ & $(1,0,\frac{1}{\sqrt{3}},\frac{1}{\sqrt{6}},\frac{1}{\sqrt{10}},0)$ \\
		 $A_5$ & $(1,\frac{1}{2},\frac{\sqrt{3}}{6},\frac{\sqrt{6}}{12},\frac{\sqrt{2}}{4},0)$ & $(1,\frac{1}{2},\frac{\sqrt{3}}{6},\frac{\sqrt{6}}{12},\frac{\sqrt{2}}{4},0)$ & $(1,-\frac{1}{2},0,0,-\frac{1}{2},0)$ & $(1,\frac{1}{2},\frac{1}{2 \sqrt{3}},\frac{1}{2 \sqrt{6}},\frac{1}{2 \sqrt{10}},0)$ \\
		 \hline
		 \multicolumn{5}{|c|}{{\bf The form $\mbox{\boldmath$u$}_i$ of sides opposite $A_i$ }}\\
		\hline
		 $\mbox{\boldmath$u$}_0$ & $(0, 0, 0, 0, 0, 1)^T$ & $(0,0,0,0,0,1)^T$ & $(0,0,0,0,0,1)^T$ & $(0,0,0,0,0,1)^T$\\
		 $\mbox{\boldmath$u$}_1$ & $(1, -2, 0, 0, 0, -1)^T$ & $(1,0,-\frac{2}{\sqrt{3}},-\frac{\sqrt{6}}{3},-\sqrt{2},-1)^T$ & $(1,0,0,0,2,-1)^T$ & $(1,-1,-\frac{1}{\sqrt{3}},-\frac{1}{\sqrt{6}},-\sqrt{\frac{5}{2}},-1)^T$ \\
		 $\mbox{\boldmath$u$}_2$ & $(0,-\sqrt{\frac{2}{3}},1,0,0,0)^T$ & $(0,-\frac{\sqrt{3}}{3},1,0,0,0)^T$ & $(0,1,1,-1,-1,0)^T$ & $(0,0,0,-1,\sqrt{\frac{5}{3}},0)^T$ \\
		 $\mbox{\boldmath$u$}_3$ & $(0,0,1,-\frac{1}{\sqrt{2}},0,0)^T$ & $(0,0,-\frac{1}{\sqrt{2}},1,0,0)^T$ & $(0,0,0,1,0,0)^T$ & $(0,0,-\frac{1}{\sqrt{2}},1,0,0)^T$ \\
		 $\mbox{\boldmath$u$}_4$ & $(0,0,0,-\sqrt{3},1,0)^T$ & $(0,0,0,-\sqrt{3},1,0)^T$ & $(0,0,-1,0,0,0)^T$ & $(0,-\frac{1}{\sqrt{3}},1,0,0,0)^T$ \\
		 $\mbox{\boldmath$u$}_5$ & $(0, 0, 0, 0, 1, 0)^T$ & $(0,1,0,0,0,0)^T$ & $(0, -1, 0, 0, 0, 0)^T$ & $(0, 1, 0, 0, 0, 0)^T$  \\
		 \hline
		 \multicolumn{5}{|c|}{{\bf Maximal horoball parameter $s_0$ }}\\
		\hline
		 $s_0$ & $0$ & $0$ & $0$ & $0$ \\
		\hline
		 \multicolumn{5}{|c|}{ {\bf Intersections $H_i = \mathcal{B}(A_0,s_0)\cap A_0A_i$ of horoballs with simplex edges}}\\
		\hline
		 $H_1$ & $(1,0,0,0,0,0)$ & $(1,0,0,0,0,0)$ & $(1,0,0,0,0,0)$ & $(1,0,0,0,0,0)$ \\
		 $H_2$ & $(1,\frac{4}{9},0,0,0,\frac{1}{9})$ & $(1,0,0,0,\frac{2 \sqrt{2}}{5},\frac{1}{5})$ & $(1,0,0,0,-\frac{4}{9},\frac{1}{9})$ & $(1,0,0,0,\frac{1}{3}\sqrt{\frac{5}{2}},\frac{1}{6})$ \\
		 $H_3$ & $(1,\frac{3}{7},\frac{\sqrt{3}}{7},0,0,\frac{1}{7})$ & $(1,0,0,\frac{2 \sqrt{6}}{19},\frac{6 \sqrt{2}}{19},\frac{3}{19})$ & $(1,0,0,\frac{2}{5},-\frac{2}{5},\frac{1}{5})$ & $(1,0,0,\frac{5}{13}\sqrt{\frac{3}{2}},\frac{3}{13}\sqrt{\frac{5}{2}},\frac{3}{13})$  \\
		 $H_4$ & $(1,\frac{8}{19},\frac{8}{19 \sqrt{3}},\frac{4}{19}\sqrt{\frac{2}{3}},0,\frac{3}{19})$ & $(1,0,\frac{4}{9 \sqrt{3}},\frac{2}{9}\sqrt{\frac{2}{3}},\frac{2 \sqrt{2}}{9},\frac{1}{9})$ & $(1,0,-\frac{2}{5},0,-\frac{2}{5},\frac{1}{5})$ & $(1,0,\frac{10}{13 \sqrt{3}},\frac{5 }{13} \sqrt{\frac{2}{3}},\frac{\sqrt{10}}{13},\frac{3}{13})$ \\
		 $H_5$ & $(1,\frac{2}{5},\frac{2}{5 \sqrt{3}},\frac{1}{5}\sqrt{\frac{2}{3}},\frac{\sqrt{2}}{5},\frac{1}{5})$ & $(1,\frac{2}{5},\frac{2}{5 \sqrt{3}},\frac{1}{5}\sqrt{\frac{2}{3}},\frac{\sqrt{2}}{5},\frac{1}{5})$ & $(1,-\frac{2}{5},0,0,-\frac{2}{5},\frac{1}{5})$ & $(1,\frac{5}{12},\frac{5}{12 \sqrt{3}},\frac{5}{12 \sqrt{6}},\frac{1}{12}\sqrt{\frac{5}{2}},\frac{1}{6})$  \\
		 \hline
		 \multicolumn{5}{|c|}{ {\bf Volume of maximal horoball piece }}\\
		\hline
		 $vol(\mathcal{B}_0 \cap \mathcal{F})$ & $0.00010\dots$ & $0.00032\dots$ & $0.00065\dots$ & $0.00116\dots$ \\
		\hline
		\multicolumn{5}{|c|}{ {\bf Optimal Packing Density}}\\
		\hline
		 $\delta_{opt}$ & $0.59421\dots$ & $0.59421\dots$ & $0.59421\dots$ & $0.56151\dots$  \\
		\hline
	\end{tabular}%
}
	\caption{Data for simply asymptotic Tilings in Cayley-Klein ball model of radius 1 centered at (1,0,0,0,0,0,0)}
	\label{table:data_one}
\end{table}

\subsection{Multiply Asymptotic Cases}


Eight Coxeter simplices have 
multiple asymptotic vertices in $\mathbb{H}^5$, see Table \ref{table:simplex_list}. In this subsection we develop a method to
find the optimal horoball packing configurations and their densities relative to 
the Coxeter tilings arising from the above simplices, when we place a 
horoball at each asymptotic vertex. 

The general horosphere equations we shall use for horospheres centered at arbitrary asymptotic vertices of such Coxeter simplices, are obtained by applying a rotation, by an element of $SO(1,5)$, to move into desired position the horosphere family centered at 
$(1,0,0,0,0,1)$ given by equation (5), we omit the details.

The main steps we follow to find the optimal horoballs configuration are:
\begin{enumerate}
\item First find bounds for the largest possible
extremal horoball type admissible at each asymptotic vertex, as in the simply asymptotic case. 
Such a horoball is tangent to the facet opposite its center, an asymptotic vertex of Coxeter 
simplex. In each case, we have at most six different extremal-type configurations depending on the 
number of the asymptotic vertices. 
\item We consider a maximal-type configuration, that is a case where one horoball is of extremal-
type (i.e. it is tangent to the opposite $4$-dimensional hyperplane). 
Then we blow up the size (type) of the remaining horoballs at the other cusps 
until they are mutually tangent or inadmissible.
\item Next we continuously ''shrink" the extremal-type horoball while simultaneously blowing up it's neighbors, and examine at each 
stage the possible locally optimal horoball configurations and their densities, i.e.
we vary the types of the horoballs within the allowable range to find the optimal 
packing density. We must keep in mind which horoballs are adjacent, and at what stages this contact structure changes. 
\item We study all possible cases and compare the finitely many locally 
optimal configurations (by Lemma \ref{lemma:szirmai}), compare their densities. Finally, we obtain the globally optimal horoball 
arrangement for a given tiling.
\end{enumerate}
The intricacy of this process for a given Coxeter tiling depends highly on the 
structure of the corresponding Coxeter simplex, its symmetries and the number of asymptotic vertices.

In what follows, we classify the possible cases and describe the above procedure according to 
the number of asymptotic vertices lying on the absolute quadric. In some cases it is possible 
to determine the optimal horoball arrangements by the symmetry properties 
of the given tiling and reducing to known optimal horoball arrangements.   
%
\subsubsection{Doubly Asymptotic cases} 

The fundamental simplices (Coxeter simplices) of the tilings 
$\overline{X}_5$, $\overline{R}_5$ and $\widehat{AU}_5$ 
each have two asymptotic vertices. 

The following lemma first proved in \cite{Sz12} gives the relationship between 
the volumes of two tangent horoball pieces 
centered at vertices of a tiling as we continuously vary their type.
If the volumes of the horoball pieces in one starting position are known,
then this lemma determines the sum of their volumes 
in all ``intermediate cases".

Let $\tau_1$ and $\tau_2$ be two congruent
$n$-dimensional convex cones with vertices at $C_1$ and $C_2$ that share common edge $\overline{C_1C_2}$.
Let $\mathcal{B}_1(x)$ and $\mathcal{B}_2(x)$ denote two horoballs centered at $C_1$ and
$C_2$ tangent at point
$I(x)\in {\overline{C_1C_2}}$. Define the point of tangency $I(0)$ (the ``midpoint") such that the
equality $V(0) = 2 vol(\mathcal{B}_1(0) \cap \tau_1) = 2 vol(\mathcal{B}_2(0) \cap \tau_2)$
holds for the volumes of the horoball sectors.

\begin{lemma}[\cite{Sz12}]
\label{lemma:szirmai}
Let $x$ be the hyperbolic distance between $I(0)$ and $I(x)$,
then
\begin{equation}
V(x) = vol(\mathcal{B}_1(x) \cap \tau_1) + vol(\mathcal{B}_2(x) \cap \tau_2) = \frac{V(0)}{2}\left( e^{(n-1)x}+e^{-(n-1)x}\right) \notag
\end{equation}
strictly increases as $x\rightarrow\pm\infty$.
\end{lemma}
\begin{proposition}
The optimal horoball packing density for Coxeter simplex tiling $\cT_{\overline{X}_5}$ 
is $\delta_{opt}(\overline{X}_5) = 0.59421\dots$
\end{proposition}
\begin{proof}
Assign coordinates to the fundamental domain $\cF_{\overline{X}_5}$ as in  
Table \ref{table:data_two}. The two asymptotic vertices are 
$A_0=(1,0,0,0,0,1)$, and $A_5=(1,\frac{1}{2},\frac{\sqrt{3}}{6},\frac{\sqrt{6}}{12},\frac{\sqrt{2}}{4},0)$. 

Let $\cF_{\overline{X}_5}$ denote the fundamental domain (Coxeter simplex) of Coxeter group
$\overline{X}_5$. 
Place two horoballs $\mathcal{B}_0\left( \arctanh s_0 \right)$ and $\mathcal{B}_5(\arctanh s_5 )$ 
that have parameters
$s_0$ and $s_5$, respectively, at $A_0$ and $A_5$.
Let $x_i = \arctanh s_i$ denote the hyperbolic distance to the center of the model 
$(1,0,0,0,0,0)$ to point $S_i=(1,0,0,0,0,s_i)$ for $i\in\{0,5\}$. 
If horoball $\mathcal{B}_0$ is of maximal type then its parameter $s_0=0$ and the horoball $\mathcal{B}_5$ when tangent has parameter $s_5=\frac{3}{5}$.
If horoball $\mathcal{B}_5$ is of maximal type then 
its parameter is $s_0=\frac{1}{3}$ and the corresponding tangent horoball $\mathcal{B}_0$ has  parameter $s_5=\frac{1}{3}$. $\mathcal{B}_0(\arctanh 0)$ and 
$\mathcal{B}_5(\arctanh\frac{1}{3})$ are tangent to hyperfaces $[\Bu_0]$ and 
$[\Bu_5]$ respectively. The packing densities the above two extremal horoball arrangements are equal to 
$\delta_{opt}$ (see Theorem 2):

\begin{align*}
\delta_{opt} &= \delta_{s_0 = 0, s_5= \frac{3}{5}}(\overline{X}_5) = \\
&= \frac{vol(\mathcal{B}_0(\arctanh 0 ) \cap \cF_{\overline{X}_5}) 
+ vol(\mathcal{B}_5(\arctanh\frac{3}{5}) \cap \cF_{\overline{X}_5})
}{vol(\cF_{\overline{X}_5})} \\
&= \delta_{s_0=\frac{1}{3},s_5=\frac{1}{3}}(\overline{X}_5) = \\
&= \frac{vol(\mathcal{B}_0(\arctanh\frac{1}{3}) \cap \cF_{\overline{X}_5}) 
+ vol(\mathcal{B}_5 (\arctanh\frac{1}{3}) \cap \cF_{\overline{X}_5})
}{vol(\cF_{\overline{X}_5})}\\
&= 0.59421\dots 
\end{align*}

Next we examine the family of horoball packing configurations between the above two maximal cases. 
Start from the horoball arrangement with parameters $s_0=0$ and $s_5=\frac{3}{5}$ where horoballs $\mathcal{B}_5(s_i)$ for $i\in\{0,5\}$ are tangent. 
Define volumes 
$V_0(x) = vol(\mathcal{B}_0(x) \cap \cF_{\overline{X}_5})$ and 
$V_5(x) = vol(\mathcal{B}_5(\arctanh \frac{3}{5}-x ) 
\cap \cF_{\overline{X}_5})$ where $x \in [0,\arctanh\frac{1}{3}]$.
Here $\arctanh \frac{1}{3}$ is the hyperbolic distance from 
$(1,0,0,0,0,0)$ to $(1,0,0,0,0,\frac{1}{3})$.

Using the formulas (2), (5), (6), (7) compute 
$V_0(0) = 0.00043\dots$ and $V_5(\arctanh \frac{3}{5}) = 0.00010\dots$. 
 
By a simple modification of Lemma \ref{lemma:szirmai},
\begin{equation}
V(x) = V_0(0) e^{-4 x}+ V_5\left(\arctanh\frac{3}{5}\right)e^{4x}.
\end{equation}
The densities of the intermediate cases between of the two 
extremal arrangements are given by
\begin{equation}
\begin{gathered}
\delta_{x}(\overline{X}_5) = \frac{vol(\mathcal{B}_0(x) \cap \cF_{\overline{X}_5}) 
+ vol(\mathcal{B}_5(\arctanh \frac{3}{5}-x ) \cap \cF_{\overline{X}_5})
}{vol(\cF_{\overline{X}_5})}=\\
= \left( \frac{4}{5} e^{-4x}+\frac{1}{5} e^{4x} \right) \delta_{opt},
\end{gathered}
\end{equation}
the maxima are attained at the endpoints of the interval domain
$[0,\arctanh \frac{1}{3}]$, see Fig \ref{fig:X5transition};
\begin{equation}
\begin{gathered}
\delta_{\arctanh\frac{1}{3}}(\overline{X}_5) = \left( \frac{4}{5} e^{-4 \arctanh\frac{1}{3}} + \frac{1}{5} e^{4 \arctanh\frac{1}{3}} \right) \delta_{opt} =\\
		= \left( \frac{4}{5} \left(\frac{1-1/3}{1+1/3}\right)^2+\frac{1}{5}\left(\frac{1+1/3}{1-1/3}\right)^2 \right)\delta_{opt}= \\
		= \left( \left(\frac{1}{2}\right)^2\frac{4}{5}+2^2\frac{1}{5} \right)\delta_{opt} = \left( \frac{1}{5}+\frac{4}{5}\right)\delta_{opt} = \delta_{opt}.
\end{gathered}
\end{equation}

\begin{figure}[h]
\begin{center}
\includegraphics[height=40mm]{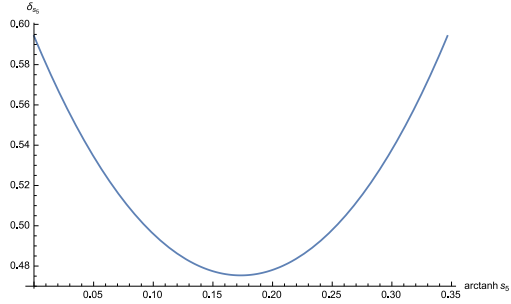}
\end{center}
\caption{The function $\delta_{s_5}$ as $s_5$ is varied in the admissible range, maxima are at the endpoints of the domain.}
\label{fig:X5transition}
\end{figure} 

The metric data, in coordinates, for the optimal horoball packings are summarized in Table \ref{table:data_two}. 
The symmetry group $\Gamma_{\overline{X}_5}$ carries the density from the 
fundamental domain to the entire tiling.

\end{proof}
\begin{corollary}
The optimal horoball packing density for the Coxeter simplex tilings 
$\Gamma \in \{\overline{R}_5, \widehat{AU}_5 \}$ are also 
$\delta_{opt}(\Gamma) = 0.59421\dots$
\end{corollary}
\begin{proof}
As in the previous theorem. 

\end{proof}
\begin{remark}
In $\overline{R}_5$ there is a unique extremal 
case, both extremal horoballs are mutually tangent. The simplex $\widehat{AU}_5$ is non-arithmetic and yields a lower optimal packing density of $0.50108\dots$. For all these cases data for the extremal packings are summarized in Table \ref{table:data_two}.
\end{remark}

\begin{landscape}

\begin{table}[h!]
\resizebox{\columnwidth}{!}{%
	\begin{tabular}{|l|l|l|l|l|}
		 \hline
		 \multicolumn{4}{|c|}{{\bf Coxeter Simplex Tilings} }\\
		 \hline
		 & \multicolumn{3}{|c|}{{\bf Doubly Asymptotic} } \\
		 \hline
		 Witt Symb.: & $\overline{X}_5$ &  $\overline{R}_5$ &  $\widehat{AU}_5$ \\
		 \hline
		 \multicolumn{4}{|c|}{{\bf Vertices of Simplex} }\\
		 \hline
		 $A_0$ & $(1,0,0,0,0,1)^*$ & $(1,0,0,0,0,1)^*$ & $(1,0,0,0,0,1)^*$ \\
		 $A_1$ & $(1,0,0,0,0,0)$ & $(1,0,0,0,0,0)$ & $(1,0,0,0,0,-1)^*$ \\
		 $A_2$ & $(1,\frac{1}{2},0,0,0,0)$ & $(1,\frac{\sqrt{2}}{2},0,0,0,0)$ & $(1,0,0,0,\frac{6 \sqrt{2+\sqrt{2}}}{13},\frac{5}{13})$ \\
		 $A_3$ & $(1,\frac{1}{2},\frac{\sqrt{3}}{6},0,0,0)$ & $(1,\frac{\sqrt{2}}{2},\frac{\sqrt{6}}{6},0,0,0)$ &$(1,0,0,\frac{1}{8} \sqrt{2+\sqrt{2}} \sqrt{3},\frac{3 \sqrt{2+\sqrt{2}}}{8},\frac{1}{2})$  \\
		 $A_4$ & $(1,\frac{1}{2},\frac{\sqrt{3}}{6},\frac{\sqrt{6}}{6},0,0)$ & $(1,\frac{\sqrt{2}}{2},\frac{\sqrt{6}}{6},\frac{\sqrt{3}}{6},0,0)$ & $(1,0,\frac{2}{35} \sqrt{12+6 \sqrt{2}},\frac{4}{35} \sqrt{2+\sqrt{2}} \sqrt{3},\frac{12 \sqrt{2+\sqrt{2}}}{35},\frac{19}{35})$ \\
		 $A_5$ & $(1,\frac{1}{2},\frac{\sqrt{3}}{6},\frac{\sqrt{6}}{6},\frac{1}{\sqrt{2}},0)*$ & $(1,\frac{\sqrt{2}}{2},\frac{\sqrt{6}}{6},\frac{\sqrt{3}}{6},\frac{1}{2},0)^*$ & $(1,\frac{3}{22} \sqrt{2} \sqrt{2+\sqrt{2}},\frac{1}{22} \sqrt{6} \sqrt{2+\sqrt{2}},\frac{1}{11} \sqrt{2+\sqrt{2}} \sqrt{3},\frac{3 \sqrt{2+\sqrt{2}}}{11},\frac{7}{11})$ \\
		 \hline
		 \multicolumn{4}{|c|}{{\bf The form $\mbox{\boldmath$u$}_i$ of sides opposite $A_i$ }}\\
		\hline
		 $\Bu_0$ & $(0,0,0,0,0,\frac{1}{24})^T$ & $(0,0,0,0,0,1)^T$ & $(1,-\frac{3 \sqrt{2}}{2 \sqrt{2+\sqrt{2}}},-\frac{\sqrt{2} \sqrt{3}}{2 \sqrt{2+\sqrt{2}}},-\frac{\sqrt{3}}{\sqrt{2+\sqrt{2}}},-\frac{3}{\sqrt{2+\sqrt{2}}},1)^T$ \\
		 $\Bu_1$ & $(-\frac{1}{24},\frac{1}{12},0,0,0,\frac{1}{24})^T$ & $(1,-\sqrt{2},0,0,0,-1)^T$ & $(1,0,0,0,-\frac{4}{3 \sqrt{2+\sqrt{2}}},-1)^T$  \\
		 $\Bu_2$ & $(0,\frac{1}{12},-\frac{1}{4 \sqrt{3}},0,0,0)^T$ & $(0,-\frac{1}{\sqrt{3}},1,0,0,0)^T$ &  $(0,0,0,-1,\frac{1}{\sqrt{3}},0)^T$ \\
		 $\Bu_3$ & $(0,0,-\frac{1}{4 \sqrt{3}},\frac{1}{4 \sqrt{6}},0,0)^T$ & $(0,0,-\frac{1}{\sqrt{2}},1,0,0)^T$ & $(0,0,-\sqrt{2},1,0,0)^T$ \\
		 $\Bu_4$ & $(0,0,0,\frac{1}{4 \sqrt{6}},-\frac{1}{12 \sqrt{2}},0)^T$ & $(0,0,0,-\sqrt{3},1,0)^T$ & $(0,-\frac{1}{\sqrt{3}},1,0,0,0)^T$  \\
		 $\Bu_5$ & $(0,0,0,0,-\frac{1}{12 \sqrt{2}},0)^T$ & $(0,0,0,0,1,0)^T$ & $(0,1,0,0,0,0)^T$  \\
		 \hline
		 \multicolumn{4}{|c|}{{\bf Maximal horoball-type parameter $s_i$ for horoball $\mathcal{B}_i$ at $A_i$ }}\\
		\hline
		 $\max s_0 \implies s_5 $ & $s_0 = 0 \implies s_5=3/5 $ & $ s_0 = 0 \implies s_5 = 3/5 $ & $s_0=\frac{1}{161} \left(73-36 \sqrt{2}\right) \implies s_1 = s_0$ \\
		 $\max s_5 \implies s_0 $ & $s_5 = 1/3 \implies s_0 = 1/3 $ & $s_5 = 3/5\implies s_0 $ & $s_1=-\frac{1}{161} \left(73-36 \sqrt{2}\right)\implies s_0=s_1$ \\
		 \hline
		 \multicolumn{4}{|c|}{ {\bf Volumes of optimal horoball pieces $V_i = vol(\mathcal{B}_i \cap \mathcal{F}_{\Gamma})$}}\\
		\hline
		 $V_{\max \mathcal{B}_0} \implies V_{\mathcal{B}_5}$ & $ 0.00043\dots \implies 0.00010\dots$ & $0.00043\dots \implies 0.00065\dots$ & $0.00368\dots \implies 0.00010\dots$ \\
		 $V_{s_0} \impliedby V_{\max s_5}$ & $0.00010\dots \impliedby 0.00043\dots$ & $0.00043\dots,0.00065\dots$ & $0.00010\dots \impliedby 0.00368\dots$ \\
		\hline
		 \multicolumn{4}{|c|}{ {\bf Dentities of horoball pieces $\delta_i = vol(\mathcal{B}_i \cap \mathcal{F}_{\Gamma})$}}\\
		\hline
		 $(\delta_{\max s_0},\delta_{s_5})$ & $(0.47537\dots=\frac{8}{10}\delta_{opt}, 0.11884\dots=\frac{2}{10}\delta_{opt})$ & $(0.23768\dots, 0.35653\dots)$ & $(0.48675\dots, 0.01432\dots)$ \\
		 $(\delta_{s_0},\delta_{\max s_5})$ & $(0.11884\dots=\frac{2}{10}\delta_{opt}, 0.47537\dots=\frac{8}{10}\delta_{opt})$ & $(0.23768\dots=\frac{4}{10}\delta_{opt}, 0.35653\dots=\frac{6}{10}\delta_{opt})$ & $(0.01432\dots, 0.48675\dots)$ \\
		\hline
		\multicolumn{4}{|c|}{ {\bf Optimal Horoball Packing Density} }\\
		\hline
		$\delta_{opt} $ & 0.59421\dots & 0.59421\dots & 0.50108\dots \\
		\hline
	\end{tabular}%
}
\caption{Data for doubly asymptotic Coxeter simplex tilings in the Cayley-Klein ball model of radius $1$ centered at $(1,0,0,0,0,0)$. Vertices marked with $^*$ are ideal.}
	\label{table:data_two}
\end{table}

\end{landscape}
%
\subsubsection{Triply Asymptotic cases} 

The fundamental simplices (Coxeter simplices) of the tilings 
$\overline{N}_5$ and $\overline{O}_5$
both have three asymptotic vertices. 

\begin{proposition}
The optimal horoball packing density for Coxeter simplex tiling $\cT_{\overline{N}_5}$ 
is $\delta_{opt}(\overline{N}_5) =  0.59421 \dots$
\end{proposition}
\begin{proof}
Assign coordinates to the fundamental domain $\cF_{\overline{N}_5}$ as in 
Table 4. The three asymptotic vertices are 
$A_0=(1,0,0,0,0,1)$, $A_2=(1,0,0,0,1,0)$, and 
$A_5=(1,\frac{\sqrt{2}}{2},0,-\frac{1}{2},\frac{1}{2},0)$.

Let $\cF_{\overline{N}_5}$ denote the fundamental domain (Coxeter simplex) of
Coxeter group $\overline{N}_5$. 
Place three horoballs $\mathcal{B}_0(\arctanh s_0)$, $\mathcal{B}_2(\arctanh s_2 )$, 
and $\mathcal{B}_5(\arctanh s_5 )$ 
with parameters
$s_0$, $s_2$, and $s_5$ at $A_0$, $A_2$, and $A_5$ respectively.
Let $x_i = \arctanh s_i $ denote the hyperbolic distance of the center of the model 
$(1,0,0,0,0,0)$ to point $S_i=(1,0,0,0,0,s_i)$ where $i\in\{0,2,5\}$. 
If the horoball $\mathcal{B}_0$ is of maximal type then $s_0 = 0$, and the tangent horoballs $\mathcal{B}_2$ and $\mathcal{B}_5$ have 
$s_2 = \frac{3}{5}$ and $s_5 = \frac{3}{5}$. 

If horoball $\mathcal{B}_2$ is of maximal type then we have the same case up the appropriate symmetry of the Coxeter simplex,
so it suffices to check densities up to the midpoint of the allowed horoball parameter range. 

If horoball $\mathcal{B}_5$ is of maximal 
type then its parameter $s_5=\frac{1}{3}$ and the corresponding tangent 
horoballs $\mathcal{B}_0$ and $\mathcal{B}_2$ have 
parameters $s_0=\frac{1}{3}$ and $s_2=\frac{7}{9}$. 
Horoballs $\mathcal{B}_0(\arctanh0)$ and 
$\mathcal{B}_5(\arctanh \frac{1}{3})$ are tangent to hyperfaces $\Bu_0$ and 
$\Bu_5$ respectively. The densities of the above two extremal horoball arrangements 
are equal to 
$\delta_{opt}$ (see Theorem 2):
\begin{align*}
\delta_{opt} &= \delta_{s_0=0,s_2=\frac{3}{5},s_5=\frac{3}{5}}(\overline{N}_5) = \\
&= \frac{vol(\mathcal{B}_0(\arctanh0) \cap \cF_{\overline{X}_5}) 
+ \sum_{i \in \{2,5\}} vol(\mathcal{B}_i(\arctanh\frac{3}{5})) \cap \cF_{\overline{X}_5})
}{vol(\cF_{\overline{N}_5})}\\
& = \delta_{s_0=\frac{1}{3},s_2=\frac{7}{9},s_5=\frac{1}{3}}(\overline{N}_5) = \\
& = \frac{ 
vol(\mathcal{B}_2(\arctanh\frac{7}{9}) \cap \cF_{\overline{N}_5})
+ \sum_{i \in \{0,5\}}
vol(\mathcal{B}_i(\arctanh\frac{1}{3}) \cap \cF_{\overline{N}_5}) 
}{vol(\cF_{\overline{N}_5})}\\
&= 0.59421\dots 
\end{align*}
We examine the horoball arrangements that transition between the above mentioned cases. Starting with the 
horoball arrangement with parameters $s_0=0$ and $s_5=\frac{3}{5}$, the horoballs 
$\mathcal{B}_5(\arctanh s_i )$ where $i\in\{0,5\}$ are tangent. Define volumes 
$V_i(x) = vol(\mathcal{B}_i(\arctanh s_i - x) \cap \cF_{\overline{N}_5})$ for $i \in \{0,2,5\}$
  with $x \in [0,\arctanh\frac{1}{3}]$ where
$\arctanh \frac{1}{3}$ is the hyperbolic distance of 
$(1,0,0,0,0,0)$ and $(1,0,0,0,0,\frac{1}{3})$.

With the formulas (2), (5), (6), and (7) compute that 
$V_0(\arctanh 0 ) = 0.23768\dots$, $V_2(\arctanh\frac{3}{5}) = 0.23768\dots$  and $V_5(\arctanh\frac{3}{5}) = 0.11884\dots$. 
 
By the simple modification of Lemma \ref{lemma:szirmai}, 
\begin{equation}
V(x) = V_0(0) e^{-4x} + V_2\left(\arctanh\frac{3}{5}\right) e^{4x} + V_5\left(\arctanh\frac{3}{5}\right) e^{4x}
\end{equation}
and the densities of the intermediate cases between of the above two 
extremal arrangements are given by the density function
\begin{equation}
\begin{split}
\delta_{x}(\overline{N}_5) &=
 \frac{vol(\mathcal{B}_0(x) \cap \cF_{\overline{N}_5}) 
+ \sum_{ i \in \{2,5\} } 
vol(\mathcal{B}_i(\arctanh \frac{3}{5}-x ) \cap \cF_{\overline{N}_5}) 
}{vol(\cF_{\overline{X}_5})}\\
&= \left( \frac{2}{5} e^{-4x}+\frac{2}{5} e^{-4x}+\frac{1}{5} e^{4x} \right) \delta_{opt}.
\end{split}
\end{equation}
where $x\in [0,\arctanh\frac{1}{3}]$.
Analysis of $\delta_{x}(\overline{N}_5)$ shows that 
 its maxima are attained at the endpoints of the interval domain 
$[0,\arctanh \frac{1}{3}]$, in particular
\begin{equation}
\begin{split}
\delta_{\arctanh\frac{1}{3}}(\overline{N}_5) 	&= \left( \frac{2}{5} e^{-4 \arctanh\frac{1}{3}} + \frac{2}{5} e^{-4 \arctanh \frac{1}{3} }+\frac{1}{5} e^{4 \arctanh\frac{1}{3}} \right) 
\delta_{opt} 
\\
& = \left( \frac{2}{5}\left(\frac{1-1/3}{1+1/3}\right)^2 + \frac{2}{5} \left(\frac{1-1/3}{1+1/3}\right)^2+\frac{1}{5} \left(\frac{1+1/3}{1-1/3}\right)^2  \right) \delta_{opt}
\\
& = \left( \left(\frac{1}{2}\right)^2\frac{2}{5}+\left(\frac{1}{2}\right)^2\frac{2}{5}+2^2\frac{1}{5} \right)\delta_{opt}
\\
& = \left( \frac{1}{10}+\frac{1}{10}+\frac{4}{5} \right)\delta_{opt}.\\
\end{split}
\end{equation}

The metric data of the optimal horoball packings are summarized in Table 4. 
The symmetry group $\Gamma_{\overline{N}_5}$ carries the density from the 
fundamental domain to the entire tiling.

\end{proof}
As the above proof we again obtain
\begin{proposition}
The optimal horoball packing density for Coxeter simplex tilings 
$\overline{O}_5$ is 
$\delta_{opt}(\overline{O}_5) = 0.59421\dots$
\end{proposition}

\begin{remark}
In the case of $\overline{O}_5$ the symmetries of the Coxeter diagram imply two classes of extremal horoball packings. 
In our notation, the horoballs at either $A_0$, $A_1$, or $A_5$ are maximized. Notice the $\mathcal{B}_0$ and $\mathcal{B}_1$ maximal cases are equivalent up to symmetry. 
As before, find horoball parameters $\max s_0 = \frac{1}{5}$, $\max s_1 = \frac{5}{7}$ and $\max s_5 = \frac{13}{19}$. By maximizing the remaining two horoballs 
subject to the constraint given the extremal horoball we find the extremal horoballs at the remaining vertices. Results are summarized in Table \ref{table:data_our}. 
In all extremal cases the packing density is $0.59421\dots$, and the relative packing densities of the three horoballs in the fundamental domain $\mathcal{F}_{\overline{O}_5}$ are given by 
$\left(\frac{3}{5},\frac{1}{5},\frac{1}{5}\right)$ and
$\left(\frac{4}{5},\frac{3}{20},\frac{1}{20}\right)$.
\end{remark}
%
\subsubsection{Doubling sequence of fundamental domains} 
%
As in the 4-dimensional case \cite{KSz14} here we have a {\it doubling sequence of fundamental domains}, meaning one simplex cell is obtained from anoother by a reflection across a given facet. 
It the $5$-dimensional case the tilings $\overline{N}_5$, $\overline{M}_5$ and $\overline{L}_5$ form a similar sequence (see Fig.~1). 
By symmetry considerations it suffices to consider the extremal horoball at the new vertex, as the other extremal cases were considered in a previous case, or are equivalent by symmetry. 

We generalize the above results to the two triply asymptotic tilings using the subgroup relations of the multiply 
asymptotic tilings given in Figure \ref{fig:lattice_of_subgroups}. The indices of the subgroups are 
\begin{equation}
|\Gamma_{\overline{N}_5}:\Gamma_{\overline{M}_5}|=
|\Gamma_{\overline{M}_5}:\Gamma_{\overline{L}_5}|=2,
\end{equation}
the fundamental domains are related by domain doubling, hence the optimal packing density is at least $\delta = 0.59421\dots$ for all multiply asymptotic cases. 
By repeated use of Lemma \ref{lemma:szirmai}, we can show that this value is indeed the optimal packing density for all multiply asymptotic 
cases. We omit the technical details of the proof. 
\begin{proposition}
The optimal horoball packing density for Coxeter simplex tilings $\cT_\Gamma$, $\Gamma \in \{ \overline{M}_5, 
\overline{L}_5 \}$ is $\delta_{opt}(\Gamma) = 0.59421\dots$.
\end{proposition}

A summary of the results for multiply asymptotic tilings are given in Table 4. 
\begin{table}[h!]
\resizebox{\columnwidth}{!}{%
	\begin{tabular}{|l|l|l|l|l|}
		 \hline
		 \multicolumn{4}{|c|}{{\bf Coxeter Simplex Tilings} }\\
		 \hline
		 & \multicolumn{3}{|c|}{{\bf Doubling Sequence} } \\
		 \hline
		 Witt Symb. & $\overline{N}_5$ &  $\overline{M}_5$ &  $\overline{L}_5$ \\
		 \hline
		 \multicolumn{4}{|c|}{{\bf Vertices of Simplex} }\\
		 \hline
		 $A_0$ & $(1,0,0,0,0,1)^*$ & $(1,0,0,0,0,1)^*$ & $(1,0,0,0,0,1)^*$ \\
		 $A_1$ & $(1,0,0,0,0,0)$ & $(1,0,0,0,0,0) \rightarrow $ & $(1,0,0,0,0,-1)^*$ \\
		 $A_2$ & $(1, 0, 0, 0, 1, 0)^*$ & $(1, 0, 0, 0, 1, 0)^*$ & $(1,0,0,0,1,0)^*$ \\
		 $A_3$ & $(\left.1,0,0,-\frac{1}{2},\frac{1}{2},0\right)$ & $(1,0,0,-\frac{1}{2},\frac{1}{2},0)$ &$(1,0,0,-\frac{1}{2},\frac{1}{2},0)$  \\
		 $A_4$ & $(1,0,\frac{1}{2 \sqrt{2}},-\frac{1}{4},\frac{3}{4},0) \rightarrow $ & $(1,0,\frac{\sqrt{2}}{2},-\frac{1}{2},\frac{1}{2},0)^*$ & $(1,0,\frac{\sqrt{2}}{2},-\frac{1}{2},\frac{1}{2},0)^*$ \\
		 $A_5$ & $(1,\frac{\sqrt{2}}{2},0,-\frac{1}{2},\frac{1}{2},0)^*$ & $(1,\frac{\sqrt{2}}{2},0,-\frac{1}{2},\frac{1}{2},0)^*$ & $(1,\frac{\sqrt{2}}{2},0,-\frac{1}{2},\frac{1}{2},0)^*$ \\
		 \hline
		 \multicolumn{4}{|c|}{{\bf The form $\mbox{\boldmath$u$}_i$ of sides opposite $A_i$ }}\\
		\hline
		 $\mbox{\boldmath$u$}_0$ & $(0, 0 , 0, 0, 0 , -1 )^T$ & $(0, 0, 0, 0, 0, 1)^T$ & $(1, 0, 0, 1, -1, 1)^T$ \\
		 $\mbox{\boldmath$u$}_1$ & $(1, 0, 0, 1, -1, -1)^T$ & $(1,0,0,0,-\sqrt{2},-1)^T$ & $(1, 0, 0, 1, -1, -1)^T$  \\
		 $\mbox{\boldmath$u$}_2$ & $(0,0,\sqrt{2},-1,-1,0)^T$ & $(0, 1, -1, -1, 1, 0)^T$ &  $(0, 0, 0, 1, 1, 0)^T$ \\
		 $\mbox{\boldmath$u$}_3$ & $(0,-1,-1,-\sqrt{2},0,0)^T$ & $(0, 0, 0, 1, 0, 0)^T$ & $(0,-1,-1,-\sqrt{2},0,0)^T$ \\
		 $\mbox{\boldmath$u$}_4$ & $(0, 0, -1, 0, 0, 0)^T$ & $(0, 0, 1, 0, 0, 0)^T$ & $(0, 0, 1, 0, 0, 0)^T$  \\
		 $\mbox{\boldmath$u$}_5$ & $(0, 1, 0, 0, 0, 0)^T$ & $(0, -1, 0, 0, 0, 0)^T$ & $(0, 1, 0, 0, 0, 0)^T$  \\
		 \hline
		 \multicolumn{4}{|c|}{{\bf Maximal horoball-type parameter $s_i$ for horoball $\mathcal{B}_i$ at $A_i$ }}\\
		\hline
		 $s_0$ & $0 (\delta=\frac{2}{5} \delta_{opt})$ & $0  (\delta=\frac{2}{5} \delta_{opt})$ & $-1/3 (\delta=\frac{4}{5} \delta_{opt})$ \\
		 $s_1$ & $-$ & $-$ & $1/3 (\delta=\frac{4}{5} \delta_{opt})$ \\
		 $s_2$ & $3/5 (\delta=\frac{2}{5} \delta_{opt})$ & $1/3 (\delta=\frac{4}{5} \delta_{opt})$ & $1/3 (\delta=\frac{4}{5} \delta_{opt})$ \\
		 $s_3$ & $-$ & $-$ & $-$ \\ 
		 $s_4$ & $-$ & $1/3 (\delta=\frac{4}{5} \delta_{opt})$ & $1/3 (\delta=\frac{4}{5} \delta_{opt})$ \\ 
		 $s_5$ & $1/3 (\delta=\frac{4}{5} \delta_{opt})$ & $1/3 (\delta=\frac{4}{5} \delta_{opt})$ & $-1/3 (\delta=\frac{4}{5} \delta_{opt})$ \\
		 \hline
		 \multicolumn{4}{|c|}{ Horoball Parameters}\\
		\hline
		 $(\mathcal{B}_0 \leftrightarrow \mathcal{B}_2) \leftrightarrow  \mathcal{B}_5$ & $({\bf s_0=0},{\bf s_2=\frac{3}{5}},s_5=\frac{3}{5})$ & - & - \\
		 $\mathcal{B}_5 \rightarrow (\mathcal{B}_0, \mathcal{B}_2)$ & $(s_0=\frac{1}{3},s_2=\frac{7}{9},{\bf s_5=\frac{1}{3}})$ & - & - \\
		 $\mathcal{B}_2 \rightarrow (\mathcal{B}_0, \mathcal{B}_4, \mathcal{B}_5)$ & - &  $(s_0=\frac{1}{3},{\bf s_2=\frac{7}{9}}, s_4=\frac{1}{3}, s_5=\frac{7}{9})$ & - \\
		 $\mathcal{B}_1 \rightarrow (\mathcal{B}_0, \mathcal{B}_2, \mathcal{B}_4, \mathcal{B}_5)$ & - &  - & $(s_0=\frac{1}{3}, {\bf s_1=\frac{-1}{3}},s_2=\frac{7}{9}, s_4=\frac{1}{3}, s_5=\frac{7}{9})$ \\

		\hline
		\multicolumn{4}{|c|}{Packing Ratios w.r.t. $\delta_{opt}$}\\
		\hline
		 $(\mathcal{B}_0 \leftrightarrow \mathcal{B}_2) \leftrightarrow  \mathcal{B}_5$ & $(\frac{2}{5},\frac{2}{5},\frac{1}{5})$ & - & - \\
		 $\mathcal{B}_5 \rightarrow (\mathcal{B}_0, \mathcal{B}_2)$ & $(\frac{4}{5},\frac{1}{10},\frac{1}{10})$ & - & - \\
		 $\mathcal{B}_2 \rightarrow(\mathcal{B}_0, \mathcal{B}_4, \mathcal{B}_5)$ & - &  $(\frac{4}{5}, \frac{1}{10},\frac{1}{20},\frac{1}{20})$ & - \\
		 $\mathcal{B}_1 \rightarrow(\mathcal{B}_0, \mathcal{B}_2, \mathcal{B}_4, \mathcal{B}_5)$ & - &  - & $(\frac{4}{5}, \frac{1}{20},\frac{1}{20},\frac{1}{20},\frac{1}{20})$ \\
		\hline
		 \multicolumn{4}{|c|}{ {\bf Optimal Horoball Packing Density} }\\
		\hline
		$\delta_{opt} $ & 0.59421\dots & $0.59421\dots$ & $0.59421\dots$\\
		\hline
	\end{tabular}%
}
	\caption{Data for multiply asymptotic Coxeter simplex tilings in the Cayley-Klein ball model of radius $1$ centered at $(1,0,0,0,0,0)$. Vertices marked with $^*$ are ideal.}
	\label{table:data_nml}
\end{table}

\subsubsection{The totally asymptotic case, $\widehat{UR}_5$}

In this case all vertices are ideal. From the symmetries of the Coxeter diagram, there are 
two classes of extremal horoball packings. In our coordinates, the horoballs at either $A_0$ or $A_2$ are to be maximized, 
the remaining four maximal cases follow from symmetry. As before, find horoball parameters $\max s_0 = \frac{1}{17}$ and $\max s_2 = \frac{133}{205}$. 
By maximizing the other five horoballs to be tangent subject to the constraints we find the largest possible 
admissible horoballs at the remaining vertices. Results are summarized in Table \ref{table:data_our}. 
The adjacency graphs of horoballs $\mathcal{B}_i$ at ideal vertices $A_i$ of the optimal cases are in Figure \ref{fig:ur5_adj}, 
observe when $\mathcal{B}_0$ is extremal $\mathcal{B}_2$ decouples from $\mathcal{B}_1$, $\mathcal{B}_3$ and $\mathcal{B}_5$. In both cases the packing density is $\delta_{opt} = 0.59421\dots$, and 
the relative packing densities of the six horoballs in the fundamental domain $\mathcal{F}_{\widehat{UR}_5}$ are given by 
$\left(\frac{3}{5},\frac{3}{20},\frac{1}{10},\frac{1}{10},\frac{1}{40},\frac{1}{40}\right)$ and
$\left(\frac{2}{5},\frac{2}{5},\frac{3}{20},\frac{3}{80},\frac{1}{160},\frac{1}{160}\right)$.
Finally we obtain the following
\begin{proposition}
The optimal horoball packing density for totally asymptotic Coxeter simplex tiling  $\cT_{\widehat{UR}_5}$ is  $\delta_{opt}(\widehat{UR}_5) \approx 0.59421$.
\end{proposition}
The proof is as in the previous cases, we omit the details. The contact structure of the optimal cases are given in Figure \ref{fig:ur5_adj}, while 
the metric data of the optimal cases are summarized in Table \ref{table:data_our}.

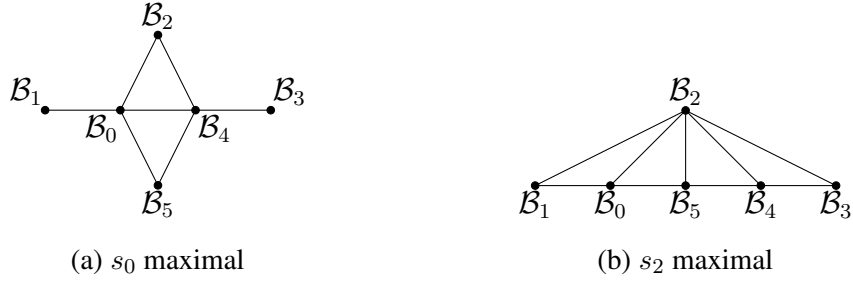
\begin{figure}[t!]
    \centering
    \begin{subfigure}[t]{0.5\textwidth}
        \centering
        \begin{tikzpicture}

	\draw [fill=black] (0,0) circle (.05);
	\draw [fill=black] (1,0) circle (.05); 
	\draw [fill=black] (2,0) circle (.05);
	\draw [fill=black] (3,0) circle (.05);
	\draw [fill=black] (1.5,1) circle (.05);
	\draw [fill=black] (1.5,-1) circle (.05);
	
	\draw (0,0) -- (3,0);
	\draw (1,0) -- (1.5,1);
	\draw (1,0) -- (1.5,-1);
	\draw (2,0) -- (1.5,1);
	\draw (2,0) -- (1.5,-1);

	\node at (-.25,.25) {$\mathcal{B}_1$};
	\node at (.75,-0.25) {$\mathcal{B}_0$};
	\node at (2.25,-0.25) {$\mathcal{B}_4$};
	\node at (3.25,0.25) {$\mathcal{B}_3$};
	\node at (1.5,1.25) {$\mathcal{B}_2$};
	\node at (1.5,-1.25) {$\mathcal{B}_5$};

    \end{tikzpicture}
        \caption{$s_0$ maximal}
    \end{subfigure}%
    ~ 
    \begin{subfigure}[t]{0.5\textwidth}
        \centering
        \begin{tikzpicture}

	\draw [fill=black] (0,0) circle (.05);
	\draw [fill=black] (1,0) circle (.05); 
	\draw [fill=black] (3,0) circle (.05);
	\draw [fill=black] (2,0) circle (.05);
	\draw [fill=black] (4,0) circle (.05);
	\draw [fill=black] (2,1) circle (.05);
	
	\draw (0,0) -- (4,0);
	\draw (0,0) -- (2,1);
	\draw (1,0) -- (2,1);
	\draw (2,0) -- (2,1);
	\draw (3,0) -- (2,1);
	\draw (4,0) -- (2,1);

	\node at (00,-0.25) {$\mathcal{B}_1$};
	\node at (1,-0.25) {$\mathcal{B}_0$};
	\node at (3,-0.25) {$\mathcal{B}_4$};
	\node at (4,-0.25) {$\mathcal{B}_3$};
	\node at (2,1.25) {$\mathcal{B}_2$};
	\node at (2,-.25) {$\mathcal{B}_5$};

    \end{tikzpicture}
        \caption{$s_2$ maximal}
    \end{subfigure}
    \caption{Horoball adjacency graphs for the two optimal packings of $\widehat{UR}_5$.}
    	\label{fig:ur5_adj}
\end{figure}

\begin{table}[h!]
\resizebox{\columnwidth}{!}{%
	\begin{tabular}{|l|l|l|}
		 \hline
		 \multicolumn{3}{|c|}{{\bf Coxeter Simplex Tilings} }\\
		 \hline
		 & \multicolumn{2}{|c|}{{\bf Index 8 Sequence} } \\
		 \hline
		 Witt Symb. &  $\overline{O}_5$ &  $\widehat{UR}_5$ \\
		 \hline
		 \multicolumn{3}{|c|}{{\bf Vertices of Simplex} }\\
		 \hline
		 $A_0$ & $(1,0,0,0,0,1)^*$ & $(1, 0, 0, 0, 0, 1)^*$ \\
		 $A_1$ & $(1,0,0,0,0,-1)^*$ & $(1, 0, 0, 0, 0, -1)^*$ \\
		 $A_2$ & $(1,0,\frac{1}{4},\frac{\sqrt{2}}{4},\frac{\sqrt{6}}{4},\frac{1}{2})$ 
		 	    & $1,0,0,0,-\frac{12}{13},\frac{5}{13})^*$ \\
		 $A_3$ & $(1,0,0,\frac{\sqrt{2}}{4},\frac{\sqrt{6}}{4},\frac{1}{2})$ 
		 	    & $(1,0,0,\frac{\sqrt{3}}{4},-\frac{3}{4},\frac{1}{2})^*$  \\
		 $A_4$ & $(1,0,0,0,\frac{\sqrt{6}}{4},\frac{1}{4})$ 
		 	    & $(1,0,\frac{4 \sqrt{6}}{35},\frac{8 \sqrt{3}}{35},\frac{24}{35},\frac{19}{35})^*$ \\
		 $A_5$ & $(1,\frac{\sqrt{3}}{4},\frac{1}{4},\frac{\sqrt{2}}{4},\frac{\sqrt{6}}{4},\frac{1}{2})^*$ 
		 	    & $(1,\frac{3 \sqrt{2}}{11},\frac{\sqrt{6}}{11},\frac{2 \sqrt{3}}{11},-\frac{6}{11},\frac{7}{11})^*$ \\
		 \hline
		 \multicolumn{3}{|c|}{{\bf The form $\mbox{\boldmath$u$}_i$ of sides opposite $A_i$ }}\\
		\hline
		 $\mbox{\boldmath$u$}_0$ & $(1,0,0,0,-\sqrt{6},1)^T$ & 
		 $(1,-\frac{3}{2 \sqrt{2}},-\frac{1}{2}\sqrt{\frac{3}{2}},-\frac{\sqrt{3}}{2},\frac{3}{2},1)^T$ \\
		 $\mbox{\boldmath$u$}_1$ & $(1,0,0,0,-\frac{\sqrt{6}}{3},-1)^T$ & $(1, 0, 0, 0, 2/3, -1)^T$  \\
		 $\mbox{\boldmath$u$}_2$ & $(0,0,0,-1,-\frac{1}{\sqrt{3}},0)^T$ & $(0,0,0,-1,-\sqrt{3}^{-1},0)$\\
		 $\mbox{\boldmath$u$}_3$ & $(0,0,-\sqrt{2},1,0,0)^T$ & $(0,0,-\sqrt{2},1,0,0)^T$ \\
		 $\mbox{\boldmath$u$}_4$ & $(0,0,0,-1,\frac{1}{\sqrt{3}},0)^T$ & $(0,-\sqrt{3}^{-1},1,0,0,0)^T$  \\
		 $\mbox{\boldmath$u$}_5$ & $(0,1,0,0,0,0)^T$ & $(0,1,0,0,0,0)^T$  \\
		 \hline
		 \multicolumn{3}{|c|}{{\bf Maximal horoball-type parameter $s_i$ for horoball $\mathcal{B}_i$ at $A_i$ }}\\
		\hline
		 $s_0$ & $s_0=\frac{1}{5}  (\delta_{s_0} =\frac{8}{10}\delta_{opt} = 0.47537\dots)$ & $\frac{1}{17} (\delta_{s_0} = \frac{4}{10}\delta_{opt} =0.23768\dots)$ \\
		 $s_1$ & $s_1=\frac{5}{7}  (\delta_{s_1}=\frac{8}{10}\delta_{opt})$ & $\frac{4}{5} (\delta_{s_1} = \frac{4}{10}\delta_{opt})$ \\
		 $s_2$ & $-$ & $\frac{133}{205} (\delta_{s_2} =\frac{6}{10}\delta_{opt}=0.35653\dots)$ \\
		 $s_3$ & $-$ & $\frac{15}{17} (\delta_{s_3} = \frac{4}{10}\delta_{opt})$ \\
		 $s_4$ & $-$ & $\frac{1153}{1297} (\delta_{s_4} =\frac{4}{10}\delta_{opt})$ \\
		 $s_5$ & $\frac{13}{19} (\delta_{s_5}=\frac{6}{10}\delta_{opt})$ & $\frac{103}{139} (\delta_{s_5} = \frac{6}{10}\delta_{opt})$ \\
		 \hline
		 \multicolumn{3}{|c|}{{\bf Progressions}}\\
		\hline
		 $\max s_0$ & $s_0=\frac{1}{5} \rightarrow s_1=\frac{1}{5} \rightarrow s_5=\frac{29}{35} (\delta = 0.59422)$ & $s_0=\frac{1}{17} \rightarrow (s_1=-\frac{1}{17}, s_2=\frac{151}{187},$ \\
		  & $\rotatebox[origin=c]{180}{$\Lsh$} s_5=\frac{29}{35} \rightarrow s_1=\frac{1}{5} (\delta= 0.59422)$ &   $s_4=\frac{1153}{1297}, s_5=233/251) \xrightarrow[]{s_4}  s_3=\frac{127}{129}$\\
		 \hline
		 $\max s_1$ & $s_1=\frac{5}{7} \rightarrow s_0=\frac{5}{7} \rightarrow s_5 = \frac{29}{35} (\delta=0.59422)$ &  see $\max s_0$\\
		  & $\rotatebox[origin=c]{180}{$\Lsh$} s_5= \frac{29}{35}\rightarrow s_1= \frac{5}{7} (\delta=0.59422)$ &  \\
		 \hline
		 $\max s_2$ & $-$ &  $s_2=\frac{133}{205} \rightarrow (s_0= \frac{5}{13}, s_1=-\frac{5}{13}, s_3=\frac{63}{65}, $\\
		 & & $s_4=\frac{1189}{1261}, s_5=\frac{56}{65}$\\
		 \hline
		 $\max s_3$ & $-$ & see $\max s_0$ \\
		 \hline
		 $\max s_4$ & $-$ & see $\max s_0$ \\
		 \hline
		 $\max s_5$ & $s_5= \frac{13}{19} \rightarrow s_0=\frac{1}{2} \rightarrow s_1=\frac{1}{2} (\delta=0.59421\dots)$ & see $\max s_2$ \\
		  & $\rotatebox[origin=c]{180}{$\Lsh$} s_1=\frac{1}{2} \rightarrow s_0=\frac{1}{2} (\delta=0.59421\dots)$ &  \\
		\hline
		 \multicolumn{3}{|c|}{ {\bf Optimal Horoball Packing Density} }\\
		\hline
		$\delta_{opt} $ & 0.59421\dots & 0.59421\dots\\
		\hline
	\end{tabular}%
	}
	\caption{Data for triply asymptotic and totally asymptotic Coxeter simplex tilings in the Cayley-Klein ball model of radius $1$ centered at $(1,0,0,0,0,0)$. Vertices marked with $^*$ are ideal.}
	\label{table:data_our}
\end{table}

\begin{figure}
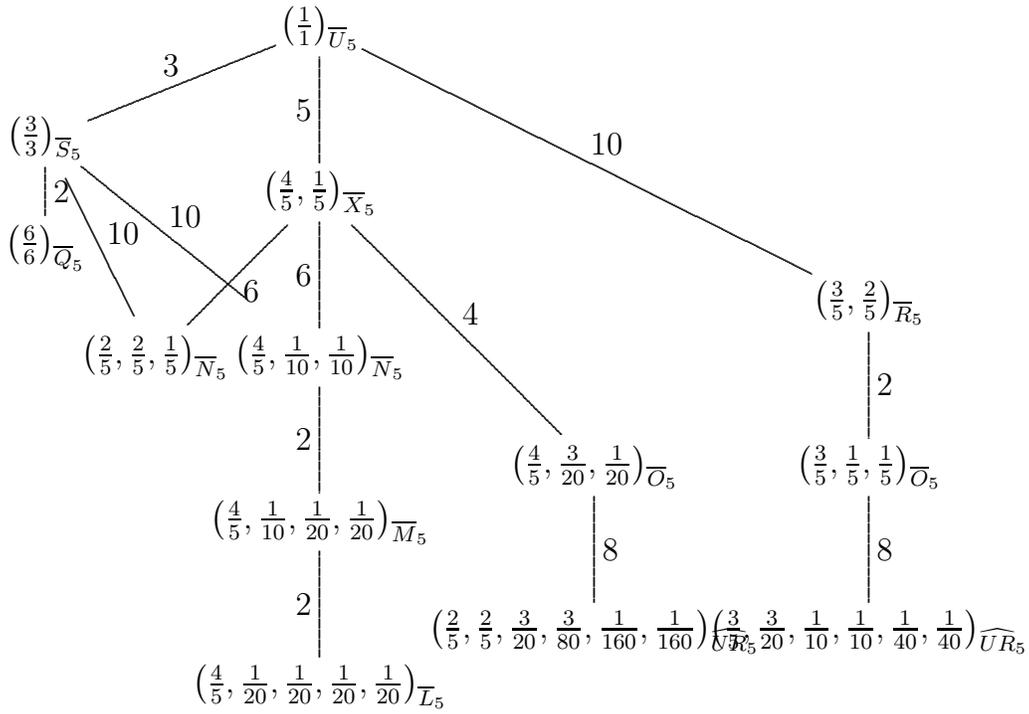

\resizebox{\columnwidth}{!}{%
\begindc{\commdiag}[200]
\obj(5,10)[uu]{$\left(\frac{1}{1}\right)_{\overline{U}_5}$}
\obj(0,8)[ss]{$\left(\frac{3}{3}\right)_{\overline{S}_5}$}
\obj(0,6)[qq]{$\left(\frac{6}{6}\right)_{\overline{Q}_5}$}
\obj(5,7)[xx]{$\left(\frac{4}{5},\frac{1}{5}\right)_{\overline{X}_5}$}
\obj(5,4)[nn]{$\left(\frac{4}{5},\frac{1}{10},\frac{1}{10}\right)_{\overline{N}_5}$}
\obj(2,4)[nn2]{$\left(\frac{2}{5},\frac{2}{5},\frac{1}{5}\right)_{\overline{N}_5}$}
\obj(5,1)[mm]{$\left(\frac{4}{5},\frac{1}{10},\frac{1}{20},\frac{1}{20}\right)_{\overline{M}_5}$}
\obj(5,-2)[ll]{$\left(\frac{4}{5},\frac{1}{20},\frac{1}{20},\frac{1}{20},\frac{1}{20}\right)_{\overline{L}_5}$}
\obj(15,5)[rr]{$\left(\frac{3}{5},\frac{2}{5}\right)_{\overline{R}_5}$}
\obj(15,2)[oo]{$\left(\frac{3}{5},\frac{1}{5},\frac{1}{5}\right)_{\overline{O}_5}$}
\obj(15,-1)[ur]{$\left(\frac{3}{5},\frac{3}{20},\frac{1}{10},\frac{1}{10},\frac{1}{40},\frac{1}{40}\right)_{\widehat{UR}_5}$}
\obj(10,2)[oo2]{$\left(\frac{4}{5},\frac{3}{20},\frac{1}{20}\right)_{\overline{O}_5}$}
\obj(10,-1)[ur2]{$\left(\frac{2}{5},\frac{2}{5},\frac{3}{20},\frac{3}{80},\frac{1}{160},\frac{1}{160}\right)_{\widehat{UR}_5}$}
\mor{uu}{ss}{$3$}[\atright, \solidline]
\mor{uu}{xx}{$5$}[\atright, \solidline]
\mor{xx}{nn}{$6$}[\atright, \solidline]
\mor{nn}{mm}{$2$}[\atright, \solidline]
\mor{mm}{ll}{$2$}[\atright, \solidline]
\mor{uu}{rr}{$10$}[\atleft, \solidline]
\mor{rr}{oo}{$2$}[\atleft, \solidline]
\mor{oo}{ur}{$8$}[\atleft, \solidline]
\mor{xx}{oo2}{$4$}[\atleft, \solidline]
\mor{oo2}{ur2}{$8$}[\atleft, \solidline]
\mor{ss}{qq}{$2$}[\atleft, \solidline]
\mor{ss}{nn}{$10$}[\atleft, \solidline]
\mor{ss}{nn2}{$10$}[\atleft, \solidline]
\mor{xx}{nn2}{$6$}[\atleft, \solidline]
\enddc%
}
\caption{Classification of the (eleven) optimal horoball packings in $\mathbb{H}^5$ by relative packing density as fraction of $\delta_{opt}$. These optimal packings are only derived from arithmetic Coxeter groups. }
\label{fig:Bifurcation diagram}
\end{figure}

\section{Conclusion}

In this paper we investigated horoball packings of asymptotic Koszul Coxeter simplex tilings of $\mathbb{H}^5$. The main result of this paper is summarized in the following theorem:

\begin{theorem}
\label{mainresult}
In $\mathbb{H}^5$ the horoball packing density $\delta_{opt}(\cT_{\Gamma}) = 0.59421\dots$ 
is realized in the ten arithmetic commensurable asymptotic Coxeter simplex tilings 
$$\Gamma \in \left\{\overline{U}_5,  \overline{S}_5, \overline{Q}_5, \overline{X}_5, \overline{R}_5, \overline{N}_5, \overline{O}_5, \overline{M}_5, \overline{L}_5, \widehat{UR}_5  \right\},$$ 
with horoballs of different types allowed at each asymptotic vertex of the tiling. 
\end{theorem}

\begin{remark}
Consider two horoball packings to be in a same class if their symmetry groups are isomorphic. 
In this sense one can distinguish between eleven different maximal horoball packings of optimal density. 
\end{remark}

The optimal packing density obtained in Theorem \ref{mainresult} is the densest ball packing of $\mathbb{H}^5$ known to the authors at the time of writing. 
However, it does not exceed the B\"or\"oczky-type simplicial upper bound for $\mathbb{H}^5$ of $0.60695\dots$. The packings we described give a new 
lower bound for the optimal ball packing density of $\mathbb{H}^5$ (see \cite{Sz12}). 

\begin{corollary}
The optimal ball packing density with horoball in same type $\delta_{opt}$ of $\mathbb{H}^5$ is bounded between
$$0.59421\dots \leq \delta_{opt} \leq 0.60695\dots.$$
\end{corollary}
\begin{remark}
We note here, that if we generalized to use horoballs of different types, then the new simplicial upper bound is greater than the B\"or\"oczky-type upper bound for $\mathbb{H}^5$.
\end{remark}
In this paper we considered the generalized simplicial densities of the horoball packings. 
{\it It is an interesting fact, that the optimal horoball packings belong to the arithmetic Coxeter groups and yield the same density. Moreover, the contributions of the individual horoballs to the global density of the packing are in rational proportions, as described in Fig.  \ref{fig:Bifurcation diagram}. The reasons of this fact are under investigation.}
It would be instructive to compare to the local Dirichlet--Voronoi 
densities of each horoball in our family of packings, and present the density of the packing as a weighted average over the cells. Agains see Fig. \ref{fig:Bifurcation diagram}
for a preliminary result. 

Results on the monotonicity of simplicial density function $d_n(r)$ for $n=5$ may help establish the optimality of our packings in $\mathbb{H}^5$ as 
in the case of $\mathbb{H}^3$ (cf. Section 1). These questions are the subject of ongoing research.



\begin{thebibliography}{}

\bibitem{Be} Bezdek,~K.
Sphere Packings Revisited,
\textit{European Journal of Combinatorics}, {\bf{27/6}} (2006), \rm 864--883.
%
\bibitem{Bo--R} Bowen,~L.~-~~Radin,~C.
Densest Packing of Equal Spheres in Hyperbolic Space,
\textit{Discrete and Computational Geometry}, {\bf{29}} (2003), \rm 23--39.
%
\bibitem{B78} B\"or\"oczky,~K.
Packing of spheres in spaces of constant curvature,
\textit{Acta Math. Acad. Sci. Hungar.}, {\bf{32}} (1978), \rm 243--261.
%
\bibitem{B--F64} B\"or\"oczky,~K. ~-~ Florian,~A.
\"Uber die dichteste Kugelpackung im hyperbolischen Raum, \textit{Acta Math. Acad. Sci. Hungar.},
{\bf{15}} (1964), \rm 237--245.
%
\bibitem{Ha} Hales,~T.~C.
Historical Overview of the Kepler Conjecture,
\textit{Discrete and Computational Geometry}, {\bf{35}} (2006), \rm 5--20.
%
\bibitem{G--K} Fejes~T\'oth,~G.~-~Kuperberg,~W.
Packing and Covering with Convex Sets,
Handbook of Convex Geometry Volume B, eds. Gruber,~P.M., Willis~J.M., pp. 799-860,
\textit{North-Holland}, (1983).
%
\bibitem{G--K--K} Fejes~T\'oth,~G.~-~Kuperberg,~G.~-~Kuperberg,~W.
Highly Saturated Packings and Reduced Coverings,
\textit{Monatshefte f\"ur Mathematik}, {\bf{125/2}} (1998), \rm 127--145.
%
\bibitem{FTL} Fejes~T\'oth,~L.
Regular Figures,
\textit{Macmillian (New York)}, 1964.
%
\bibitem{IH90} Im Hof,~H.-C. Napier cycles and hyperbolic Coxeter groups,
\textit{Bull. Soc. Math. Belgique}, {\bf{42}} (1990), 523--545.
%
\bibitem{JKRT} Johnson,~N.W., Kellerhals,~R., Ratcliffe,~J.G., Tschants,~S.T.  
The Size of a Hyperbolic Coxeter Simplex,
\textit{Transformation Groups}, {\bf{4/4}} (1999), \rm 329--353.
%
\bibitem{JKRT2} Johnson,~N.W., Kellerhals,~R., Ratcliffe,~J.G., Tschants,~S.T.  
Commensurability classes of hyperbolic Coxeter Groups,
\textit{Linear Algebra and its Applications}, {\bf{345}} (2002), \rm. 119--147.
%
\bibitem{JW} Johnson,~N.W., Weiss,~A.I.  
Quaternionic Modular Groups,
\textit{Linear Algebra and its Applications}, {\bf{295/1-3}} (1999), \rm 159--189.
%
\bibitem{K91} Kellerhals,~R.  
The Dilogarithm and Volumes of Hyperbolic Polytopes,
\textit{AMS Mathematical Surveys and Monographs}, {\bf{37}} (1991), \rm 301--336.
%
\bibitem{K98} Kellerhals,~R. 
Ball packings in spaces of constant curvature and the simplicial density function,
\textit{Journal f\"ur reine und angewandte Mathematik}, {\bf{494}} (1998), \rm 189--203.
%
\bibitem{KSz} Kozma,~R.T., Szirmai,~J.  
Optimally dense packings for fully asymptotic Coxeter tilings by horoballs of different types,
\textit{Monatshefte f\"ur Mathematik}, {\bf{168/1}} (2012), \rm 27--47.
%
\bibitem{KSz14} Kozma,~R.T.~---~Szirmai,~J.
New Lower Bound for the Optimal Ball Packing Density of Hyperbolic 4-space,
\emph{Discrete Comput. Geom.}, {\bf 53/1} (2015), 182-198, DOI: 10.1007/s00454-014-9634-1.
%
\bibitem{Ma99} Marshall,~T.~H. 
Asymptotic Volume Formulae and Hyperbolic Ball Packing,
\textit{Annales Academic Scientiarum Fennica: Mathematica}, {\bf{24}} (1999), \rm 31--43.
%
\bibitem{Mol97} Moln\'ar,~E. 
The Projective Interpretation of the eight 3-dimensional homogeneous geometries,
\textit{Beitr. Algebra Geom.,}, {\bf{38/2}} (1997), \rm 261--288.
%
\bibitem{MSz} \text{Moln\'ar,~E.~-~Szirmai,~J.}
Symmetries in the 8 homogeneous 3-geometries, 
\textit{Symmetry Cult. Sci.},  \textbf{21/1-3} (2010), 87-117.
%
\bibitem{R06} Radin,~C.
The symmetry of optimally dense packings,
Non-Eucledian Geometries, eds. A.~Pr\'ekopa, E.~Moln\'ar, pp. 197-207,
\textit{Springer Verlag}, (2006).
%
\bibitem{Ro64} Rogers,~C.A.
Packing and Covering,
Cambridge Tracts in Mathematics and Mathematical Physics 54,
\textit{Cambridge University Press}, (1964).
%
\bibitem{Sz05-2} Szirmai,~J. 
The optimal ball and horoball packings of the Coxeter tilings in the hyperbolic 3-space
\textit{Beitr. Algebra Geom.,} {\bf{46/2}} (2005),  545--558.
%
\bibitem{Sz07-1} Szirmai,~J. 
The optimal ball and horoball packings to the Coxeter honeycombs in the hyperbolic d-space
\textit{Beitr. Algebra Geom.,} {\bf{48/1}} (2007), 35--47.
%
\bibitem{Sz07-2} Szirmai,~J. 
The densest geodesic ball packing by a type of Nil lattices,
\emph{Beitr. Algebra Geom.,}
{\bf{48/2}} (2007), 383--397.
%
\bibitem{Sz10}Szirmai,~J. 
The densest translation  ball packing by fundamental lattices in Sol space,
\emph{Beitr. Algebra Geom.,} {\bf 51/2}  (2010), 353--373. 
%
\bibitem{Sz12} Szirmai,~J. 
Horoball packings to the totally asymptotic regular simplex in the hyperbolic n-space,  
\emph{Aequationes mathematicae},  {\bf 85} (2013), 471-482,  DOI: 10.1007/s00010-012-0158-6.
%
\bibitem{Sz12-2}Szirmai,~J.
Horoball packings and their densities by generalized simplicial density function in the hyperbolic space,
\emph{Acta Math. Hung.,}
{\bf 136/1-2} (2012), 39--55, DOI: 10.1007/s10474-012-0205-8.
%
\bibitem{Sz13-1} Szirmai,~J. 
Regular prism tilings in ${\mathbf{SL(2,R)}}$ space,  
\emph{Aequationes mathematicae}, (2013), DOI: 10.1007/s00010-013-0221-y.  
%
\bibitem{Sz13-2} Szirmai,~J. 
Simply transitive geodesic ball packings to $\mathbf{S^2\times R}$ space groups generated by glide reflections,
{\emph {Annali di Matematica Pura ed Applicata}}, (2013), DOI: 10.1007/s10231-013-0324-z.
%
\end{thebibliography}
\end{document}